\documentclass{amsproc}

\usepackage{amssymb}
\usepackage{float}

\usepackage{ytableau}
\usepackage{tabularx}
\usepackage{tikz}
\usetikzlibrary{arrows,positioning}
\usepackage{enumitem}
\usepackage{comment}
\usepackage{url}

\newcommand{\N}{\mathbb{N}}

\newcommand{\Z}{\mathbb{Z}}

\newcommand{\set}[1]{\left\{#1 \right\}}

\DeclareMathOperator{\F}{F}
\DeclareMathOperator{\G}{G}
\DeclareMathOperator{\m}{m}
\DeclareMathOperator{\g}{g}
\DeclareMathOperator{\Hook}{H}
\DeclareMathOperator{\PF}{PF}
\DeclareMathOperator{\Pa}{P}
\DeclareMathOperator{\q}{q}
\DeclareMathOperator{\type}{t}
\DeclareMathOperator{\A}{A}
\DeclareMathOperator{\M}{M}

\newcommand{\hook}{\text{hook}}

\newtheorem{theorem}{Theorem}[section]

\newtheorem{lemma}[theorem]{Lemma}
\newtheorem{corollary}[theorem]{Corollary}
\newtheorem{prop}[theorem]{Proposition}

\theoremstyle{definition}
\newtheorem{definition}[theorem]{Definition}
\newtheorem{example}[theorem]{Example}

\newtheorem{question}{Question}

\theoremstyle{remark}

\numberwithin{equation}{section}

\begin{document}

\title{On the Smallest Partition Associated to a Numerical Semigroup}

\author[N. Kaplan]{Nathan Kaplan}
\address{Department of Mathematics, University of California, Irvine, 340 Rowland Hall, Irvine, CA 92697-3875}
\email{nckaplan@math.uci.edu}

\thanks{The authors thank the referee for many helpful suggestions that significantly improved the paper.  We thank Ian Farish, Erik Imathiu-Jones, Matilda LaFortune, and Victoria Wiest for many helpful discussions. The first author thanks Maria Monks Gillespie for conversations many years ago that helped get this project started. The authors received support from NSF Grant DMS 2154223. The second and third author received support from the UCI UROP Summer Undergraduate Research Program.}

\author[K. Kim]{Kaylee Kim}
\address{Department of Mathematics, University of California, Irvine, 340 Rowland Hall, Irvine, CA 92697-3875}
\email{kayleesk@uci.edu}

\author[C. McGeorge]{Cole McGeorge}
\address{Department of Mathematics, University of California, Irvine, 340 Rowland Hall, Irvine, CA 92697-3875}
\email{cmcgeorg@uci.edu}

\author[F. Ramirez]{Fabian Ramirez}
\address{Department of Mathematics, University of California, Irvine, 340 Rowland Hall, Irvine, CA 92697-3875}
\email{fabiar3@uci.edu}

\author[D. Singhal]{Deepesh Singhal}
\address{Department of Mathematics, University of California, Irvine, 340 Rowland Hall, Irvine, CA 92697-3875}
\email{singhald@uci.edu}


\date{\today}

\begin{abstract}
The set of hook lengths of an integer partition $\lambda$ is the complement of some numerical semigroup $S$.  There has been recent interest in studying the number of partitions with a given set of hook lengths.  Very little is known about the distribution of sizes of this finite set of partitions.  We focus on the problem of determining the size of the smallest partition with its set of hook lengths equal to $\N\setminus S$.
\end{abstract}

\maketitle

\section{Introduction}

An \emph{integer partition} $\lambda$ is a list of positive integers $(\lambda_1,\lambda_2,\ldots, \lambda_\ell)$ satisfying $\lambda_1 \ge \lambda_2 \ge \cdots \ge \lambda_\ell$.  The $\lambda_i$ are called the \emph{parts} of $\lambda$.  We write $|\lambda| = \lambda_1 + \cdots + \lambda_\ell$ for the \emph{size} of $\lambda$.  If $|\lambda| = n$ we say that $\lambda$ is a partition of $n$.  We can represent a partition $\lambda$ using its \emph{Young diagram}, a left-justified array of boxes with $\lambda_i$ boxes in the $i$\textsuperscript{th} row.
Every box has an associated \emph{hook}, which is the set of boxes below and to the right of the box (including the original box).
The \emph{hook length} of a box is the size of its hook, that is, the number of boxes below it, plus the number of boxes to the right of it, plus $1$ for the box itself.
The \emph{hook set} of $\lambda$, denoted $\Hook(\lambda)$, is the set of hook lengths of the Young diagram of $\lambda$.  Ignoring the left edge and the top edge, the boundary of the Young diagram of $\lambda$ is the \emph{profile} of $\lambda$.  We give an example in Figure \ref{Fig:hook_length}.

Our next goal is to explain the connection between hook sets of partitions and numerical semigroups.  We first need to discuss numerical sets.  Let $\N = \{0,1,2,\ldots\}$ denote the set of nonnegative integers.  A \emph{numerical set} $T$ is a subset of $\N$ that contains $0$ and has finite complement in $\N$.  The elements of $\N\setminus T$ are called the \emph{gaps} of $T$, denoted $\G(T)$.  The largest of these gaps is the \emph{Frobenius number} of $T$, denoted $\F(T)$, and the number of these gaps is the \emph{genus} of $T$, denoted $\g(T)$.  The smallest nonzero element of $T$ is the \emph{multiplicity} of $T$, denoted $\m(T)$.  The Frobenius number of $T$ lies in the interval between two consecutive multiples of $\m(T)$.  That is, there is a unique nonnegative integer $\q(T)$, called the \emph{depth} of $T$, such that $(\q(T)-1)\m(T) \le \F(T) < \q(T)\m(T)$. 

There is a bijection between the set of integer partitions and the set of numerical sets given by the enumeration of a numerical set $T$.  We construct this map by describing a walk starting at the origin in $\Z^2$ that is determined by the elements of $T$.
Starting with $n = 0$:
\begin{itemize}
\item if $n \in T$, draw a line of unit length to the right,
\item if $n \not\in T$, draw a line of unit length up,
\item repeat for $n+1$.
\end{itemize}
For any integer $n > \F(T)$ we draw a line to the right.  We see that this process ends with an infinite series of steps to the right.  We disregard this section of the walk and see that we have constructed the profile of the Young diagram of a partition.  This is the partition $\lambda(T)$, called the \emph{enumeration of $T$}.  This idea of studying numerical sets and numerical semigroups by associating them with walks on a grid was introduced by Bras-Amor\'os and de Mier \cite{BAdM}.  For a more extensive discussion of this bijection see \cite[Section 2]{CONSTANTIN201799}.

Let $T$ be a numerical set. Note that each box of $\lambda(T)$ is uniquely determined by a pair $(u,x)$ satisfying $u\in T,\  x\in \G(T)$ and $u<x$. We denote this box by $\Box(u,x)$ and denote its hook by $\hook(u,x)$. 
The size of a hook is the number of boxes in it. Note that $\hook(u,x)$ has size $|\hook(u,x)|=x-u$.
We also see that
\begin{equation}\label{eqn: formlua size partition}
|\lambda(T)|=\#\{(u,x)\colon u\in T, x\in \G(T), u<x\}=
\sum_{x\in \G(T)}\#\{u\in T\colon u<x\}. 
\end{equation}

When we list the elements of a numerical set $T$ we always write them in increasing order.  Since the complement of $T$ is finite, we can write $T = \{0, a_1, ..., a_k, ~\rightarrow~\}$, where the symbol $\rightarrow$ indicates that $T$ contains all integers greater than $a_k$.  We see that the partition given in Figure \ref{Fig:hook_length} is the enumeration of the numerical set $T = \{0,5,7,9,\rightarrow\}$.
\begin{figure}[h]
    \centering
    \begin{minipage}{0.3\textwidth}
        \centering
        \begin{tikzpicture}[scale=0.7,y=0.9cm,x=1cm]
            \coordinate (A) at (0,0);
            \coordinate (B) at (1,0);
            \coordinate (C) at (1,1);
            \coordinate (D) at (1,2);
            \coordinate (E) at (1,3);
            \coordinate (F) at (1,4);
            \coordinate (G) at (2,4);
            \coordinate (H) at (2,5);
            \coordinate (I) at (3,5);
            \coordinate (J) at (3,6);
            \coordinate (K) at (4,6);
            \coordinate (L) at (5,6);
            \coordinate (M) at (6,6);
            \coordinate (N) at (7,6);

            \draw[->] (A)--(B);
            \draw[->] (B)--(C);
            \draw[->] (C)--(D);
            \draw[->] (D)--(E);
            \draw[->] (E)--(F);
            \draw[->] (F)--(G);
            \draw[->] (G)--(H);
            \draw[->] (H)--(I);
            \draw[->] (I)--(J);
            \draw[->] (J)--(K);
            \draw[->] (K)--(L);
            \draw[dotted, thick] (L)--(M);

            \node at (0.5,-0.3) {$0$};
            \node at (1.25,0.5) {$1$};
            \node at (1.25,1.5) {$2$};
            \node at (1.25,2.5) {$3$};
            \node at (1.25,3.5) {4};
            \node at (1.6,3.7) {5};
            \node at (2.25,4.5) {6};
            \node at (2.6,4.7) {7};
            \node at (3.25,5.5) {8};
            \node at (3.6,5.7) {9};
            \node at (4.6,5.7) {10};

        \end{tikzpicture}
    \end{minipage}
    \hfill
    \begin{minipage}{0.3\textwidth}
        \centering
            \begin{tikzpicture}[scale=0.7, ,y=0.9cm,x=1cm]
        \coordinate (A) at (0,0);
        \coordinate (B) at (1,0);
        \coordinate (C) at (1,1);
        \coordinate (D) at (1,2);
        \coordinate (E) at (1,3);
        \coordinate (F) at (1,4);
        \coordinate (G) at (2,4);
        \coordinate (H) at (2,5);
        \coordinate (I) at (3,5);
        \coordinate (J) at (3,6);
        \coordinate (N) at (0,6);
        \coordinate (c) at (0,1);
        \coordinate (d) at (0,2);
        \coordinate (e) at (0,3);
        \coordinate (f) at (0,4);
        \coordinate (h) at (0,5);
        \coordinate (b) at (1,6);
        \coordinate (g) at (2,6);

        \draw (A)--(B)--(C)--(D)--(E)--(F)--(G)--(H)--(I)--(J)--(N)--(A);
        \draw (c)--(C);
        \draw (d)--(D);
        \draw (e)--(E);
        \draw (f)--(F);
        \draw (h)--(H);
        \draw (b)--(B);
        \draw (g)--(G);

        \node at (0.5,-0.3) {0};
        \node at (1.25,0.5) {1};
        \node at (1.25,1.5) {2};
        \node at (1.25,2.5) {3};
        \node at (1.25,3.5) {4};
        \node at (1.6,3.7) {5};
        \node at (2.25,4.5) {6};
        \node at (2.6,4.7) {7};
        \node at (3.25,5.5) {8};

    \end{tikzpicture}
    \hspace{.2 in}
    \end{minipage}
    \begin{minipage}{.3\textwidth}
            \begin{tikzpicture}[scale=0.7, ,y=0.9cm,x=1cm]
        \coordinate (A) at (0,0);
        \coordinate (B) at (1,0);
        \coordinate (C) at (1,1);
        \coordinate (D) at (1,2);
        \coordinate (E) at (1,3);
        \coordinate (F) at (1,4);
        \coordinate (G) at (2,4);
        \coordinate (H) at (2,5);
        \coordinate (I) at (3,5);
        \coordinate (J) at (3,6);
        \coordinate (N) at (0,6);
        \coordinate (c) at (0,1);
        \coordinate (d) at (0,2);
        \coordinate (e) at (0,3);
        \coordinate (f) at (0,4);
        \coordinate (h) at (0,5);
        \coordinate (b) at (1,6);
        \coordinate (g) at (2,6);

        \draw (A)--(B)--(C)--(D)--(E)--(F)--(G)--(H)--(I)--(J)--(N)--(A);
        \draw (c)--(C);
        \draw (d)--(D);
        \draw (e)--(E);
        \draw (f)--(F);
        \draw (h)--(H);
        \draw (b)--(B);
        \draw (g)--(G);

        \node at (0.5,-0.4) {};

        \node at (0.5,0.5) {$1$};
        \node at (0.5,1.5) {$2$};
        \node at (0.5,2.5) {$3$};
        \node at (0.5,3.5) {$4$};
        \node at (0.5,4.5) {$6$};
        \node at (0.5,5.5) {$8$};
        \node at (1.5,4.5) {$1$};
        \node at (1.5,5.5) {$3$};
        \node at (2.5,5.5) {$1$};

    \end{tikzpicture}
    \end{minipage}
    \caption{From left to right we have, the walk defined by $T=\{0,5,7,9,\rightarrow\}$, the Young diagram of $\lambda(T)$, and $\lambda(T)$ where each box is labeled with its hook length.}
    \label{Fig:hook_length}
\end{figure}
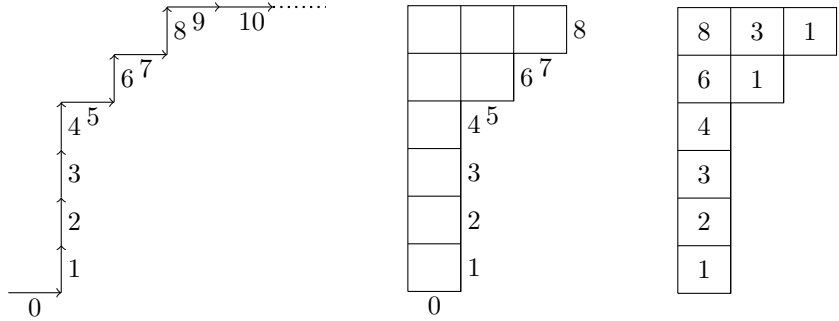

A numerical set $S$ that is closed under addition is a \emph{numerical semigroup}.  It is not difficult to show that the hook set of an integer partition $\lambda$ is the complement of some numerical semigroup $S$.  See for example \cite{KeithNath} or \cite[Proposition 4]{CONSTANTIN201799}.  Several authors have recently studied families of integer partitions that arise from special classes of numerical semigroups.  See for example \cite{BNST, ArfPartition}.  Let $\mathcal{P}(S) = \{\lambda\colon \Hook(\lambda) = \N \setminus S\}$ and $\Pa(S) = |\mathcal{P}(S)|$. There has been recent interest in studying $\Pa(S)$; see \cite{chen2023enumeratingnumericalsetsassociated,ChenAsymptoticGrowth,CONSTANTIN201799,doi:10.1080/00927872.2021.1918136}. However, very little is known about the distribution of the sizes of the partitions in $\mathcal{P}(S)$.  We focus on a particular question about these sizes.
\begin{question}\label{Q:1}
Let $S$ be a numerical semigroup.  What is the size of the smallest partition $\lambda$ with $\Hook(\lambda) = \N \setminus S$?
\end{question}

We first observe that since $\Hook(\lambda(S)) = \N \setminus S$, the set $\mathcal{P}(S)$ is nonempty.  See \cite[Proposition 4]{CONSTANTIN201799} or \cite[Proposition 1]{MarzuolaMiller} for a proof.  If $S$ is a numerical semigroup for which $\lambda(S)$ has minimal size among all partitions in $\mathcal{P}(S)$, then we say that $S$ is \emph{enumeration-minimal} or \emph{$\lambda$-minimal}.  Analyzing small examples suggests that it is very common for a semigroup $S$ to have this property.
We implemented \cite[Algorithm 5.1]{chen2023enumeratingnumericalsetsassociated} using the \texttt{numericalsgps} package \cite{NumericalSgps} in the computer algebra system GAP \cite{GAP} to verify the following result.  Computations for this paper were also done in the computer algebra system Sage \cite{sagemath10.6beta8}. Our code is available at~\cite{DDeepuS2025LambdaMinimality}.
\begin{prop}
\label{prop:small minimal}
Let $S$ be a numerical semigroup.  If $\F(S) \le 16$ or $\g(S) \le 11$, then $S$ is $\lambda$-minimal.
\end{prop}

One can use the \texttt{numericalsgps} package to check that there are $784$ numerical semigroups $S$ with $\F(S) \le 16$ and there are $820$ numerical semigroups with $\g(S)~\le~11$. 
 However, not every numerical semigroup is $\lambda$-minimal.
\begin{example}\label{ex:non_lam}
Let $S = \{0,9,10,11,12,13,18,\rightarrow\}$.  Then $|\lambda(S)| = 32$.  The partition $\lambda = (9,8,2,2,2,2,2,2,2)$ also has $\Hook(\lambda) = \N \setminus S$, and $|\lambda| = 31$.
\end{example}
\noindent This paper came out of an attempt to understand this example.  Further computation shows that the semigroup $S$ from Example \ref{ex:non_lam} is the only one of the $592$ numerical semigroups of genus $12$ that is not $\lambda$-minimal.  

We now describe several of our main results.  We first introduce some additional notation.  A set of positive integers $n_1,\ldots, n_t$ is a \emph{generating set} for a numerical semigroup $S$ if $S$ is equal to the set of linear combinations of $n_1,\ldots, n_t$ with nonnegative integer coefficients.  That is, 
\[
S = \langle n_1,\ldots, n_t \rangle= \{x\in \mathbb{Z} \colon x = a_{1}n_{1}+ ... + a_{t}n_{t} \text{ with } a_{i} \in \N\}.
\]
It is not difficult to show that every numerical semigroup $S$ has a unique generating set of minimal size.  We direct the reader to \cite{rosales_numerical_2009} for a general reference on numerical semigroups.

We show that the semigroup $S$ from Example \ref{ex:non_lam} is a specialization of two different infinite families of numerical semigroups that are not $\lambda$-minimal.
\begin{theorem}\label{thm:inf_non_lam}
Let $m \ge 9$ and $S_{m} = \langle m, m+1,\ldots, 2m-5\rangle$.  Then $S_m$ is not $\lambda$-minimal.
\end{theorem}
We prove this result in Section \ref{sec:non-lam} by determining the sizes of all of the partitions in $\mathcal{P}(S_m)$.

\begin{theorem}\label{thm:2k1_3k1_intro}
Let $k \ge 4$ be a positive integer that is not prime.  Then $S_k = \langle 2k+1, 2k+2, \ldots, 3k+1\rangle$ is not $\lambda$-minimal.
\end{theorem}
In fact, we show that if $k$ has many proper divisors, then there are many partitions in $\mathcal{P}(S_k)$ with size less than $|\lambda(S_k)|$.

In Sections \ref{sec:staircase}, \ref{sec:depth2}, and \ref{sec:small_type} we prove that several families of numerical semigroups $S$ are $\lambda$-minimal.  
\begin{theorem}\label{thm:staircase}
Let $m$ and $k$ be positive integers with $m \ge 2$ and let $s$ be an integer satisfying $1 \le s \le m-1$. The numerical semigroup
\[
S_{m,k,s} = \{0,m,2m, \ldots, km, km+s+1, \rightarrow \}
\]
is $\lambda$-minimal.
\end{theorem}
This theorem is the focus of Section \ref{sec:staircase}. 
These numerical semigroups have previously appeared in the literature.  See, for example, \cite[Corollary 4.5]{casabella2024apery}, where the authors compute the size of their ideal class monoids. In \cite{ChenAsymptoticGrowth}, the authors show that for a fixed $m$ and $s$, the cardinality of $\mathcal{P}(S_{m,k,s})$ is given by a polynomial in $k$.

 In Section \ref{sec:depth2} we prove that several types of numerical semigroups with depth~$2$ are $\lambda$-minimal. A partition $\lambda$ has a \emph{Durfee square} of size $k$ if $k$ is the largest positive integer $i$ such that $\lambda$ has at least $i$ parts of size greater than or equal to~$i$.  Equivalently, $\lambda$ has a Durfee square of size $k$ if $k$ is the largest side length of a square contained in the Young diagram of $\lambda$.  For example, the partition in Figure \ref{Fig:hook_length} has a Durfee square of size $2$.
\begin{theorem}\label{thm:Durfee}
Let $S$ be a numerical semigroup of depth $2$.  If $\lambda(S)$ has a Durfee square of size at most $3$, then $S$ is $\lambda$-minimal.
\end{theorem}
We note that if $S$ is the semigroup from Example \ref{ex:non_lam}, then $S$ has depth $2$, $\lambda(S)$ has a Durfee square of size $4$, and $S$ is not $\lambda$-minimal.

We use the term \emph{small elements} of a numerical semigroup $S$ for the positive elements of $S$ that are less than $\F(S)$. Some authors include $0$ and also the smallest element of $S$ larger than $\F(S)$ in the set of small elements of $S$.  The set we call the small elements of $S$ together with $0$ is sometimes called the set of \emph{sporadic elements} of $S$ or the set of \emph{left elements} of $S$.

The number of small elements of $S$ is one less than the number of columns of the Young diagram of $\lambda(S)$.
If $S$ is a numerical semigroup with $k$ small elements, then $\lambda(S)$ has a Durfee square of size at most $k+1$. Therefore, Theorem \ref{thm:Durfee} implies that a numerical semigroup of depth $2$ with at most $2$ small elements is $\lambda$-minimal.  We prove something stronger.

\begin{theorem}\label{thm:small_elements}
Let $S$ be a numerical semigroup of depth $2$.  If $S$ has at most $3$ small elements, then $S$ is $\lambda$-minimal.
\end{theorem} 
The numerical semigroup in Example~\ref{ex:non_lam} has $5$ small elements and is not $\lambda$-minimal.
We suspect that all numerical semigroups of depth $2$ with $4$ small elements are $\lambda$-minimal, although we cannot prove this. Moreover, we suspect that for any $n$, there are only finitely many numerical semigroups of depth $2$ that have $n$ small elements and are not $\lambda$-minimal.

In Section \ref{sec:small_type}, we focus on numerical semigroups of small type.  The idea of considering numerical semigroups in terms of this invariant goes back to an influential paper of Fr\"oberg, Gottlieb, and H\"aggkvist \cite{Fröberg198663}.  The set of \emph{pseudo-Frobenius numbers} of $S$ is
\[
\PF(S) = \{P \in \N \setminus S \colon P + s \in S \text{ for all } s \in S \setminus \{0\}\}.
\]
The number of pseudo-Frobenius numbers of $S$ is the \emph{type} of $S$, denoted $\type(S)$.  At the beginning of Section \ref{sec:small_type}, we explain how earlier results imply that numerical semigroups of type $1$ or $2$ are $\lambda$-minimal. 
We extend this to numerical semigroups of type~$3$.
\begin{theorem}
\label{thm:type3minimality}
Let $S$ be a numerical semigroup with $\type(S) = 3$.  Then $S$ is $\lambda$-minimal.
\end{theorem}

It is not possible to extend this kind of result to semigroups of larger type.  We will show in Section \ref{sec:non-lam} that the semigroups $S_m$ from Theorem \ref{thm:inf_non_lam} all have $\type(S_m) = 4$, so there are infinitely many  semigroups of type $4$ that are not $\lambda$-minimal. We also note that the semigroups $S_{5,k,s}$ in Theorem \ref{thm:staircase} have type $4$, so there are infinitely many semigroups of type $4$ that are $\lambda$-minimal.

We describe the connection between partitions with a given set of hook lengths and numerical sets with a given atom monoid. Antokoletz and Miller in \cite{ANTOKOLETZ2002636} define the \emph{atom monoid} of a numerical set $T$ as 
\[
\A(T) = \{x \in \N \colon x + T \subseteq T\}.
\]
The atom monoid $A(T)$ is sometimes denoted as $T-T$ or $T:T$. 
It is not difficult to see that $\A(T) \subseteq T$ and that $\A(T)$ is closed under addition.  Therefore, $\A(T)$ is a numerical semigroup contained in $T$.  Moreover, if $S$ is a numerical semigroup, then $\A(S) = S$.  If $T$ is a numerical set with $\A(T) = S$, we say that $T$ is a \emph{numerical set associated to $S$}.  Numerical sets associated to $S$ are in bijection with partitions $\lambda$ with $\Hook(\lambda) = \N \setminus S$ \cite[Corollary 1]{KeithNath}.
In fact, $\mathcal{P}(S)=\{\lambda(T): \A(T)=S\}$.
This correspondence is also discussed in \cite[Section 2]{CONSTANTIN201799}.  Marzuola and Miller introduced the \emph{Anti-Atom Problem}, which asks for the number of numerical sets associated to a given numerical set $S$ \cite{MarzuolaMiller}.  This is equivalent to computing $\Pa(S)$.

The set of \emph{normalized ideals} of $S$ is
\[
\mathfrak{J}_0(S)=\{T\subseteq \mathbb{N}\colon  0\in T,\ T+S\subseteq T\}.
\]
For a discussion of normalized ideals of numerical semigroups, see \cite[Section 3]{BonzioGarciaSanchez}.  Further results in this direction are found in \cite{barucci2016class,casabella2024apery}. 
Note that all numerical sets associated to $S$ are normalized ideals of $S$. In fact a numerical set $T$ is a normalized ideal of $S$ if and only if $S\subseteq A(T)$.

The main focus of \cite{chen2023enumeratingnumericalsetsassociated} is on computing $\Pa(S)$. 
The results about $\Pa(S)$ proven in that paper are all phrased in terms of numerical sets associated to a numerical semigroup. 
The authors study a poset defined in terms of a subset of the gaps of $S$ and show that numerical sets associated to $S$ correspond to certain order ideals of this poset. 
In Section \ref{sec:void}, we prove that a particular class of order ideals corresponding to numerical sets associated to $S$ cannot produce partitions of size smaller than $|\lambda(S)|$.

\section{The Void Poset and Order Ideals}\label{sec:void}

A main result of \cite{chen2023enumeratingnumericalsetsassociated} is that numerical sets associated to a numerical semigroup $S$ correspond to certain order ideals of a poset constructed from a subset of the gaps of $S$. 
\begin{definition}[Void and void poset]
    Let $S$ be a numerical semigroup. The \emph{void} of $S$ is the set
    \[
    \M(S) = \{x\in \G(S) \colon \F(S)-x \in \G(S)\}.
    \]
The partial order on $\M(S)$ is defined as follows. Given $x, y \in \M(S)$, we say $x \preccurlyeq y$ if and only if $y-x \in S$. The partially ordered set $(\M(S),\preccurlyeq)$ is the \emph{void poset} of $S$.
\end{definition}
\noindent The void of a numerical semigroup $S$ is called the set of \textit{missing pairs} of $S$ in \cite{CONSTANTIN201799}. 
The elements of the void are referred to as \emph{holes of second type} in \cite{patil2007type}, as \emph{$h$-gaps} in \cite{aicardi2010gaps}, and simply as \emph{holes} in the \texttt{numericalsgps} pacakage \cite{NumericalSgps}.

We recall a basic property of $(\M(S),\preccurlyeq)$.  
\begin{prop}\cite[Lemma 2.1 and Proposition 2.2]{chen2023enumeratingnumericalsetsassociated}\label{prop:max_min_elements}
The maximal elements of $(\M(S),\preccurlyeq)$ are the elements of $\PF(S) \setminus \{\F(S)\}$ and the minimal elements of $(\M(S),\preccurlyeq)$ are the elements $\F(S) - P$ where $P \in \PF(S) \setminus \{\F(S)\}$.
\end{prop}
Therefore, we can determine $\type(S)$ by counting the maximal elements of $(\M(S),\preccurlyeq~)$ and adding one for $\F(S)$.

\begin{definition}[The dual of a numerical set]
Let $T$ be a numerical set.  The \emph{dual} of $T$ is 
\[
T^* = \{x\in \Z \colon \F(T) - x\not\in T\}.
\]
\end{definition}
Marzuola and Miller in \cite[Proposition 1]{MarzuolaMiller} show that if $T$ is a numerical set associated to $S$, then $S \subseteq T \subseteq S^*$. It is shown in \cite[Lemma 3]{CONSTANTIN201799} that $S^* = S \cup \M(S)$.
The numerical set $S^*$ is also called the \emph{canonical ideal} of $S$. These results imply that numerical sets associated to $S$ are of the form $S \cup I$ where $I \subseteq \M(S)$.  

A numerical semigroup $S$ is called symmetric if $s\in S$ if and only if $\F(S)-s\notin S$. It is easy to see that a numerical is symmetric if and only if it coincides with its canonical ideal. This is equivalent to $\M(S)=\varnothing$. In fact, symmetric numerical semigroups are the only ones for which $\type(S)=1$ \cite{Fröberg198663}. They are also the only numerical semigroups for which $\Pa(S)=1$.

A main result of \cite{chen2023enumeratingnumericalsetsassociated} is to characterize the subsets $I\subseteq \M(S)$ that give numerical sets associated to $S$.

\begin{definition}[Order ideal]
    Let $S$ be a numerical semigroup and $I \subseteq \M(S)$. If $x \in I$ and $x \preccurlyeq y$ implies $y \in I$, then $I$ is an \emph{order ideal} of $(\M(S),\preccurlyeq)$.
\end{definition}

The following result gives an important property of $\A(S\cup I)$ for $I \subseteq \M(S)$.

\begin{prop}\cite[Proposition 2.4]{chen2023enumeratingnumericalsetsassociated}
\label{semigroup subset}
    Let $S$ be a numerical semigroup and $I \subseteq \M(S)$. Then $S \subseteq \A(S\cup I)$ if and only if $I$ is an order ideal of $(\M(S), \preccurlyeq)$.
\end{prop}
This proposition says that a subset $I\subseteq \M(S)$ is an order ideal of $(\M(S),\preccurlyeq)$ if and only if $S\cup I$ is a normalized ideal of $S$. In order to characterize numerical sets associated to $S$, one must determine the conditions on an order ideal $I \subseteq \M(S)$ that imply $\A(S\cup I) \subseteq S$.

\begin{definition}[Ideal triangle]
Let $P,x,y \in \M(S)$ with $P+x+y = \F(S)$. We call $(P,x,y)$ an \emph{ideal triangle} of $S$.

Given an order ideal $I \subseteq \M(S)$ and an ideal triangle $(P,x,y)$ we say that $I$ \emph{satisfies} $(P,x,y)$ if $P, x\in I$ and $\F(S)-y\notin I$.
\end{definition}

The notion of a \emph{Frobenius triangle} introduced in \cite[Definition 3.1]{chen2023enumeratingnumericalsetsassociated} is a special case of an ideal triangle. 

\begin{definition}[Frobenius triangle]
An ideal triangle $(P,x,y)$ in which $P\in \PF(S)$ is called a \emph{Frobenius triangle} of $S$.
\end{definition}

We recall the characterization of numerical sets associated to a numerical semigroup $S$.
\begin{theorem}\cite[Theorem 3.9]{chen2023enumeratingnumericalsetsassociated}\label{thm:classify_num_set}
Let $S$ be a numerical semigroup, $I \subseteq \M(S)$, and $T = S \cup I$.  Then $T$ is a numerical set associated to $S$ if and only if 
\begin{enumerate}[leftmargin=*]
\item $I$ is an order ideal of $(\M(S),\preccurlyeq)$, and
\item for each $P \in I \cap \PF(S)$, one of the following conditions is satisfied: $2P\not\in S$, $\F(S) - P \in I$, or there is a Frobenius triangle $(P,x,y)$ that is satisfied by $I$.
\end{enumerate}
Moreover, if $S \cup I$ is a numerical set associated to $S$, then for all $P \in I \cap \PF(S)$ either $\F(S)-P \in I$ or there is a Frobenius triangle $(P,x,y)$ that is satisfied by $I$.
\end{theorem}

The \emph{special gaps} of $S$ are the gaps $h\in \G(S)$ for which $\{h\}\cup S$ is a numerical semigroup.  It is easy to see that the special gaps of $S$ are precisely the $P\in\PF(S)$ for which $2P\in S$. For more information, see \cite[Chapter 3]{rosales_numerical_2009}. Thus, condition (2) of Theorem~\ref{thm:classify_num_set} says that all special gaps $P$ of $S$ that are in $I$ satisfy: either $\F(S)-P\in I$, or $I$ satisfies some Frobenius triangle $(P,x,y)$.

In \cite{chen2023enumeratingnumericalsetsassociated} it is shown that there is a class of order ideals of $(\M(S),\preccurlyeq)$ that always satisfy the hypotheses of Theorem~\ref{thm:classify_num_set}.
\begin{definition}[Self-dual order ideal]
    Let $S$ be a numerical semigroup and $I\subseteq \M(S)$ be an order ideal of $(\M(S),\preccurlyeq)$. If $x \in I$ implies $\F(S) -x \in I$, then $I$ is called \emph{self-dual}.
\end{definition}

\begin{prop}\cite[Proposition 3.2]{chen2023enumeratingnumericalsetsassociated}
If $S$ is a numerical semigroup and $I$ is a self-dual order ideal of $(\M(S),\preccurlyeq)$, then $S \cup I$ is a numerical set associated to $S$.
\end{prop}

Self-dual order ideals of $(\M(S),\preccurlyeq)$ give numerical sets associated to $S$, but we now show that the partition associated to such a numerical set cannot have size smaller than $|\lambda(S)|$.
\begin{prop}
\label{self-dual-implies-lambda-S-small}
    Let $S$ be a numerical semigroup, and let $T = S \cup I$ be a numerical set associated to $S$ where $I$ is a self-dual order ideal of $(\M(S),\preccurlyeq)$. Then $|\lambda(S)|\leq|\lambda(T)|$.
\end{prop}

We apply the following lemma in the proof of this proposition.
\begin{lemma}
    \label{set_counting}
    Let $S$ be a numerical semigroup and let $I$ be an order ideal of $(\M(S),\preccurlyeq)$. Suppose $T= S \cup I$ is a numerical set associated to $S$. Then 
    \[
    |\lambda(T)| = |\lambda(S)| + |A| - |B|,
    \]
    where 
    \[
    A=\{(i,h)\colon i\in I,\ h \in \G(T),\ i<h\}\quad\text{and}\quad B= \{(s,i)\colon s\in S,\ i\in I,\ s<i\}.
    \]
\end{lemma}

\begin{proof}
    \sloppy
We recall the discussion in the paragraph ending with equation \eqref{eqn: formlua size partition}.  Each box in the Young diagram of $\lambda(T)$ is $\Box(t,h)$ for some $t<h$, with $t \in T$ and $h \in \G(T)$. 
We can determine $|\lambda(T)|$ by counting such pairs.  We have
    \begin{eqnarray*}
    |\lambda(S)| & = & \#\{(s, g) \colon s \in S,\ g\in \G(S),\ s<g\} \\
    |\lambda(T)| & = & \#\{(t, h) \colon t \in T,\  h\in \G(T),\ t<h\}.
    \end{eqnarray*}
    Since $T= S\cup I$ for some $I\subseteq \M(S)$, we have  
    \[
    |\lambda(T)| = \#\{(s,h)\colon s\in S,\ h\in \G(T),\ s<h\} + \#\{(i, h)\colon i\in I,\ h\in \G(T),\ i<h\}.
    \]
That is, 
    \[ 
    |\lambda(T)| = \#\{(s,h)\colon s\in S,\ h\in \G(T),\ s<h\} + |A|.
    \]
We note that $\G(T) =  \G(S) \setminus I$. This implies
    \[
    |\lambda(T)| = \#\{(s,g)\colon s\in S,\ g\in \G(S),\ s<g\} - \#\{(s, i)\colon s\in S,\ i\in I,\ s<i\} + |A|,
    \]
completing the proof.
\end{proof}

\begin{proof}[Proof of Proposition \ref{self-dual-implies-lambda-S-small}]
    Let $\F(S) = F$ and define the sets $A$ and $B$ as in the statement of Lemma \ref{set_counting}. Consider the function $f \colon B\rightarrow A$ defined by $f(s,i) = (F-i,F-s)$. Because $I$ is self-dual and $i \in I$, we know that $F-i\in I$. Since $s \in S$, it is clear that $F-s \in \G(S) \setminus \M(S)$. This implies $F-s \notin T$. Since $s< i$, we have $F-i < F-s$. So $f$ is well-defined. It is easy to see that $f$ is injective. Hence, $|\lambda(S)|\leq |\lambda(T)|$.
\end{proof}

Theorem \ref{thm:classify_num_set} gives a characterization of the order ideals $I$ of $(\M(S),\preccurlyeq)$ for which $T = S\cup I$ is a numerical set associated to $S$.  In particular, such an order ideal gives a numerical set corresponding to $S$ if and only if for each $P\in I\cap \PF(S)$ one of two conditions is satisfied.  We show that a similar statement holds for other elements of $I$ that are not necessarily pseudo-Frobenius numbers of $S$.
\begin{prop}
    \label{Element in I}
    Let $S$ be a numerical semigroup and $I$ be an order ideal of $(\M(S),\preccurlyeq)$ such that $T = S\cup I$ is a numerical set associated to $S$. 
For each $x \in I$ one of the following conditions holds:
    \begin{enumerate}
        \item $\F(S)-x\in I$, or
        \item there is an ideal triangle $(x,y,z)$ satisfied by $I$.
    \end{enumerate}
\end{prop}
\begin{proof}
Let $\F(S)=F$. Since the elements of $\PF(S)\setminus\{F\}$ are the maximal elements of $(\M(S),\preccurlyeq)$, we know that there is some $P\in \PF(S)\setminus\{F\}$ such that $x\preccurlyeq P$. Since $I$ is an order ideal, we know that $P\in I$. Theorem~\ref{thm:classify_num_set} implies that either $F-P\in I$ or $I$ satisfies a Frobenius triangle of the form $(P,y,z)$.
\begin{enumerate}[leftmargin=*]
\item Suppose $F-P\in I$. Note that $x\preccurlyeq P$ implies $F-P\preccurlyeq F-x$, so $F-x\in I$.

\item Suppose $I$ satisfies a Frobenius triangle $(P,y,z)$. Let $s=P-x$. We know that $s\in S$ since $x\preccurlyeq P$. Let $y_1=y+s$. We want to show that $y_1\in \M(S)$. 

First, if $y_1\in S$, then $(F-z)-x= (P+y)-x=y+s=y_1$. This means $x\preccurlyeq F-z$. However, $x\in I$ and $F-z\notin I$. This contradicts the fact that $I$ is an order ideal. Thus, $y_1\notin S$. 

On the other hand, if $F-y_1\in S$, then $F-y= (F-y_1)+s\in S$. This contradicts the fact that $y\in \M(S)$. Thus, $F-y_1\notin S$.

We conclude that $y_1\in \M(S)$ and hence $(x,y_1,z)$ is an ideal triangle. Since $y\in I$, we know that $y_1=y+s\in I$. Also, $F-z\notin I$ since $I$ satisfies the Frobenius triangle $(P,y,z)$. Therefore, $I$ satisfies the ideal triangle $(x,y_1,z)$.
\end{enumerate}

\end{proof}

\section{Numerical Semigroups that are not $\lambda$-minimal}\label{sec:non-lam}

The focus of this section is on numerical semigroups that are not $\lambda$-minimal.  We first return to the semigroup $S$ from Example \ref{ex:non_lam}.
\begin{example}\label{ex:non_lam_I}
     Let $S=\langle 9,10,11,12,13\rangle$ and consider the order ideal $I=\{1,14,16\} \subseteq \M(S)$. It is not difficult to check that $T=S\cup I$ is a numerical set associated to $S$ for which 
     \[
     |\lambda(S)|=32 > 31=|\lambda(T)|.
     \] 
     Hence, $S$ is not $\lambda$-minimal.
\end{example}

We now consider an infinite family of numerical semigroups containing the semigroup $S = \langle 9,10,11,12,13\rangle$.
\begin{lemma}\label{Frob}
    Let $m \ge 9$ and $S=\langle m,m+1,m+2,...,2m-5\rangle$. Then 
    \[
    \G(S)  =   \{1,2,\ldots, m-1, 2m-4,2m-3,2m-2,2m-1\}.
    \]
    \indent We see that
    \[
    \PF(S)  =   \{2m-4,2m-3,2m-2,2m-1\},
    \]
    and in particular, $\F(S) = 2m-1$.
\end{lemma}

\begin{proof}
The smallest element of $S$ that is not one of its minimal generators is $2m$, so it is clear that $\{1,2,\ldots, m-1, 2m-4,2m-3,2m-2,2m-1\} \subseteq \G(S)$.
Since $m\geq 9$, we have $3m-1\leq (2m-5)+(2m-5)$. Adding pairs of minimal generators shows that $S$ contains $2m, 2m+1,\ldots, 3m-5, 3m-4, 3m-3, 3m-2, 3m-1$.  Thus, $S$ contains $m$ consecutive positive integers. Adding multiples of $m$ shows that $S$ contains all positive integers larger than $2m$.  Therefore, 
\[
\G(S) = \{1,2,\ldots, m-1, 2m-4,2m-3,2m-2,2m-1\}.
\] 
Once we know $\G(S)$ it is easy to determine $\PF(S)$.
\end{proof}  

Before proving Theorem \ref{thm:inf_non_lam} we introduce some terminology related to integer partitions.
\begin{definition}
Let $\lambda = (\lambda_1,\ldots, \lambda_\ell)$ be a partition.  The conjugate $\tilde{\lambda}$ of $\lambda$ is $(\lambda_1', \ldots, \lambda_r')$ where $\lambda_i' = \#\{j\colon \lambda_j \ge i\}$. A partition $\lambda$ is \emph{self-conjugate} if $\lambda = \tilde{\lambda}$.
\end{definition}
It is easy to see that the Young diagram of $\tilde{\lambda}$ has columns of length $\lambda_1, \lambda_2,\ldots, \lambda_\ell$.  It is clear that $|\lambda| = |\tilde{\lambda}|$, and moreover that the multiset of hook lengths of $\lambda$ is the same as the multiset of hook lengths of $\tilde{\lambda}$.

We recall a result about the relationship between the enumeration of $T$ and the enumeration of $T^*$.
\begin{prop}\cite[Proposition 12]{CONSTANTIN201799}\label{prop:enum_Tdual}
For any numerical set $T$, $\widetilde{\lambda(T)}=~\lambda(T^*)$, and therefore $\A(T)=\A(T^*)$.
\end{prop}

Theorem \ref{thm:inf_non_lam} follows from the following more precise result.
\begin{theorem}\label{thm:inf_non_lam_sizes}
Let $m \ge 9$ and $S = \langle m, m+1,\ldots, 2m-5\rangle$.  Then we have  $|\lambda(S)| = 5m - 13$, $\Pa(S) = 6$, and 
\[ 
\{|\lambda(T)|: T \in \mathcal{P}(S)\} =
\{4m-5, 5m-13, 5m-7\}.
\]  
That is, $S$ is not $\lambda$-minimal. In fact, we have  
\[
\lim_{m\rightarrow \infty} \left(|\lambda(S)| - \min\{|\lambda(T)| \colon T \in \mathcal{P}(S)\}\right) = \infty
\]
and 
\[
\lim_{m\rightarrow\infty}\frac{\min\{|\lambda(T)|\colon T \in \mathcal{P}(S)\}}{|\lambda(S)|} = \frac{4}{5}.
\]
\end{theorem}

\begin{proof}
The description of $\G(S)$ from Lemma \ref{Frob} implies that 
    \[
    \M(S)=\{1,2,3,2m-4,2m-3,2m-2\}.
    \]

Theorem \ref{thm:classify_num_set} implies that the only order ideals $I$ of $(\M(S),\preccurlyeq)$ that give numerical sets associated to $S$ are
    \[
    \varnothing,\ \M(S), \ \{1,2m-4,2m-2\},\ \{2,2m-4,2m-3\},\ \{1,2m-4\},\ \{1,2,2m-4,2m-3\}.
    \]
    For each of these order ideals we consider $T = S\cup I$ and construct the Young diagram of $\lambda(T)$.  We will see that these order ideals come in conjugate pairs.

    The first pair of order ideals is $\varnothing$ and $\M(S)$.  The Young diagrams of the numerical sets corresponding to these order ideals are:
    
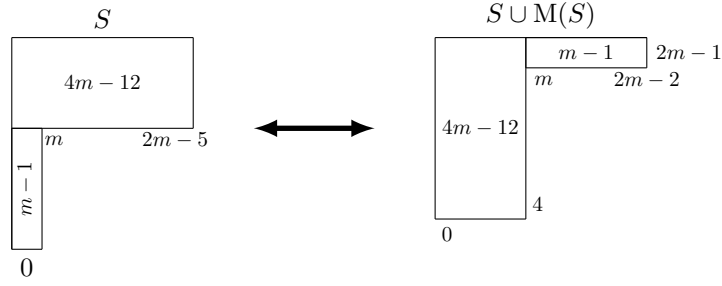
\begin{center}
        \begin{tikzpicture}[scale=0.8]
        
            \coordinate (A) at (0,0);
            \coordinate (B) at (0,2);
            \coordinate (C) at (0,3.5);
            \coordinate (D) at (.5,0);
            \coordinate (E) at (.5,2);
            \coordinate (F) at (3,2);
            \coordinate (G) at (3,3.5);
            \coordinate (H) at (0.23,1);

            \draw (A) -- (C);
            \draw (A) -- (B);
            \draw (A) -- (D) node[midway, below] {0};
            \draw (D) -- (E);
            \draw (B) -- (E);
            \draw (B) -- (F);
            \draw (F) -- (G);
            \draw (G) -- (C) node[midway, above] {$S$};

            \node[scale=0.8] at (1.5,2.75) {$4m-12$};
            \node[scale=0.8] at (0.7,1.8) {$m$};
            \node[scale=0.8] at (2.7,1.8) {$2m-5$};
            \node[scale=0.8] at (H) {\rotatebox{90}{$m-1$}};

            \coordinate (I) at (4,2);
            \coordinate (J) at (6,2);

            \draw [>=latex, <->, line width=2pt, scale=2] (I) -- (J);

            \coordinate (K) at (10.5,3.5);
            \coordinate (L) at (8.5,3.5);
            \coordinate (M) at (7,3.5);
            \coordinate (N) at (10.5,3);
            \coordinate (O) at (8.5,3);
            \coordinate (P) at (8.5,0.5);
            \coordinate (Q) at (7,0.5);
            \coordinate (R) at (9.5,3.27);

            \draw (K) -- (M) node[midway, above] {$S\cup \M(S)$};
            \draw (K) -- (L);
            \draw (K) -- (N);
            \draw (N) -- (O);
            \draw (L) -- (O);
            \draw (L) -- (P);
            \draw (P) -- (Q);
            \draw (Q) -- (M);

            \node[scale=0.8] at (7.75,2) {$4m-12$};
            \node[scale=0.8] at (7.2,0.25) {0};
            \node[scale=0.8] at (8.7,0.75) {4};
            \node[scale=0.8] at (8.8, 2.8) {$m$};
            \node[scale=0.8] at (10.5,2.8) {$2m-2$};
            \node[scale=0.8] at (11.2,3.25) {$2m-1$};
            \node[scale=0.8] at (9.5,3.25) {$m-1$};

        \end{tikzpicture}
    \end{center}
\noindent We compute $|\lambda(S)|=|\lambda(S\cup \M(S))|=5m-13.$ 

The next pair of order ideals is $\set{1, 2m-4, 2m-2}$ and $\set{2, 2m-4, 2m-3}$.  The Young diagrams of the numerical sets corresponding to these order ideals are:
    
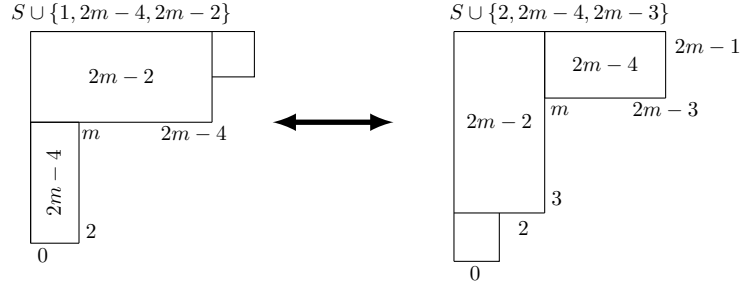
\begin{center}
        \begin{tikzpicture}[scale=0.8]
        
            \coordinate (A) at (0,0);
            \coordinate (B) at (0,2);
            \coordinate (C) at (0,3.5);
            \coordinate (D) at (.8,0);
            \coordinate (E) at (.8,2);
            \coordinate (F) at (3,2);
            \coordinate (G) at (3,3.5);
            \coordinate (H) at (0.4,1);
            \coordinate (1) at (3,2.75);
            \coordinate (2) at (3.7,2.75);
            \coordinate (3) at (3.7,3.5);

            \draw (A) -- (C);
            \draw (A) -- (B);
            \draw (A) -- (D);
            \draw (D) -- (E);
            \draw (B) -- (E);
            \draw (B) -- (F);
            \draw (F) -- (G);
            \draw (G) -- (C) node[midway, above,scale=0.8] {$S\cup \{1,2m-4,2m-2\}$};
            \draw (1) -- (2) -- (3) -- (C);

            \node[scale=0.8] at (1.5,2.75) {$2m-2$};
            \node[scale=0.8] at (1,1.8) {$m$};
            \node[scale=0.8] at (2.7,1.8) {$2m-4$};
            \node[scale=0.8] at (H) {\rotatebox{90}{$2m-4$}};
            \node[scale=0.8] at (.2,-.2) {0};
            \node[scale=0.8] at (1,.2) {2};

            \coordinate (I) at (4,2);
            \coordinate (J) at (6,2);

            \draw [>=latex, <->, line width=2pt, scale=2] (I) -- (J);

            \coordinate (K) at (10.5,3.5);
            \coordinate (L) at (8.5,3.5);
            \coordinate (M) at (7,3.5);
            \coordinate (N) at (10.5,2.4);
            \coordinate (O) at (8.5,2.4);
            \coordinate (P) at (8.5,0.5);
            \coordinate (Q) at (7,0.5);
            \coordinate (R) at (9.5,3.27);
            \coordinate (11) at (7,-.3);
            \coordinate (22) at (7.75,-.3);
            \coordinate (33) at (7.75,0.5);

            \draw (K) -- (M) node[midway, above, scale=0.8] {$S\cup \{2,2m-4,2m-3\}$};
            \draw (K) -- (L);
            \draw (K) -- (N);
            \draw (N) -- (O);
            \draw (L) -- (O);
            \draw (L) -- (P);
            \draw (P) -- (Q);
            \draw (Q) -- (M);
            \draw (11) -- (Q);
            \draw (11) -- (22) -- (33);

            \node[scale=0.8] at (7.75,2) {$2m-2$};
            \node[scale=0.8] at (7.35,-.5) {0};
            \node[scale=0.8] at (8.15,.25) {2};
            \node[scale=0.8] at (8.7,0.75) {3};
            \node[scale=0.8] at (8.75, 2.2) {$m$};
            \node[scale=0.8] at (10.5,2.2) {$2m-3$};
            \node[scale=0.8] at (11.2,3.25) {$2m-1$};
            \node[scale=0.8] at (9.5,2.95) {$2m-4$};

        \end{tikzpicture}
    \end{center}
We compute $|\lambda(S\cup \{1,2m-4,2m-2\})|=|\lambda(S\cup\{2,2m-4,2m-3\})|=4m-5$. 

Lastly, we consider the pair of order ideals $\set{1, 2m-4}$ and $\set{1, 2, 2m-4, 2m-3}$.  The Young diagrams of the numerical sets corresponding to these order ideals are:
    
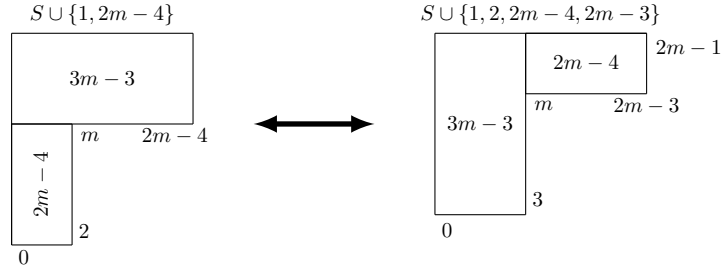
\begin{center}
        \begin{tikzpicture}[scale=0.8]
        
            \coordinate (A) at (0,0);
            \coordinate (B) at (0,2);
            \coordinate (C) at (0,3.5);
            \coordinate (D) at (1,0);
            \coordinate (E) at (1,2);
            \coordinate (F) at (3,2);
            \coordinate (G) at (3,3.5);
            \coordinate (H) at (0.23,1);

            \draw (A) -- (C);
            \draw (A) -- (B);
            \draw (A) -- (D);
            \draw (D) -- (E);
            \draw (B) -- (E);
            \draw (B) -- (F);
            \draw (F) -- (G);
            \draw (G) -- (C) node[midway, above, scale=0.8] {$S\cup\{1,2m-4\}$};

            \node[scale=0.8] at (1.5,2.75) {$3m-3$};
            \node[scale=0.8] at (1.3,1.8) {$m$};
            \node[scale=0.8] at (2.7,1.8) {$2m-4$};
            \node[scale=0.8] at (.5,1) {\rotatebox{90}{$2m-4$}};
            \node[scale=0.8] at (.2,-.2) {0};
            \node[scale=0.8] at (1.2,.23) {2};

            \coordinate (I) at (4,2);
            \coordinate (J) at (6,2);

            \draw [>=latex, <->, line width=2pt, scale=2] (I) -- (J);

            \coordinate (K) at (10.5,3.5);
            \coordinate (L) at (8.5,3.5);
            \coordinate (M) at (7,3.5);
            \coordinate (N) at (10.5,2.5);
            \coordinate (O) at (8.5,2.5);
            \coordinate (P) at (8.5,0.5);
            \coordinate (Q) at (7,0.5);
            \coordinate (R) at (9.5,3.27);

            \draw (K) -- (M) node[midway, above, scale=0.8] {$S\cup\{1,2,2m-4,2m-3\}$};
            \draw (K) -- (L);
            \draw (K) -- (N);
            \draw (N) -- (O);
            \draw (L) -- (O);
            \draw (L) -- (P);
            \draw (P) -- (Q);
            \draw (Q) -- (M);

            \node[scale=0.8] at (7.75,2) {$3m-3$};
            \node[scale=0.8] at (7.2,0.25) {0};
            \node[scale=0.8] at (8.7,0.75) {3};
            \node[scale=0.8] at (8.8, 2.3) {$m$};
            \node[scale=0.8] at (10.5,2.3) {$2m-3$};
            \node[scale=0.8] at (11.2,3.25) {$2m-1$};
            \node[scale=0.8] at (9.5,3) {$2m-4$};

        \end{tikzpicture}
    \end{center}
We compute $|\lambda(S\cup \{1,2m-4\})|=|\lambda(S\cup\{1,2,2m-4,2m-3\})|= 5m-7$. The final statement of the theorem follows by taking a limit as $m \rightarrow \infty$.
\end{proof}

The semigroup $S = \langle 9,10,11,12,13\rangle$ from Example \ref{ex:non_lam} also generalizes in a different direction to give a second infinite family of numerical semigroups that are not $\lambda$-minimal.
\begin{theorem}\label{thm:2k1_3k1}
Let $k \ge 4$ and $S = \langle 2k+1,2k+2\dots,3k+1\rangle$.  Then $|\lambda(S)|=k^2+4k$.

Let $l$ be a divisor of $k$.  Consider 
\[
T=S\cup \{1,2,\dots,l-1\}\cup \{3k+1<x\leq4k+1 \colon x\not\equiv 1\pmod{l}\}.
\] 
Then $\A(T) = S$ and 
\[
|\lambda(T)|=\frac{3k^2l-k^2+4kl^3+3kl^2+kl-2l^4+2l^3}{2l^2}.
\]
\end{theorem}
\begin{proof}
Note that $\G(S)=\{1,2,\dots,2k\}\cup \{3k+2,\dots,4k+1\}$. This implies $\F(S) = 4k+1$ and $\M(S)=\{1,2,\dots,k-1\}\cup \{3k+2,\dots,4k\}$.

By equation~\eqref{eqn: formlua size partition}, we have 
\[
|\lambda(S)| = \sum_{y\in \G(S)} \#\{s\in S\colon s<y\} = 2k\cdot 1 +k(k+2)=k^2+4k.
\]
Similarly, 
\begin{align*}
\G(T)&=\{l,l+1,\dots,2k\}\cup \{3k+1<y \le 4k+1\colon y\equiv 1\pmod{l}\}\\
&=\{l,l+1,\dots,2k\} \cup \{3k+1+ly_1\colon 1\leq y_1\leq k/l\}.
\end{align*}
So
\begin{align*}
|\lambda(T)| &= \sum_{y\in \G(T)} \#\{x\in T\colon x<y\}\\
&= (2k-(l-1))l +\sum_{y_1=1}^{k/l} (l+(k+1)+(l-1) y_1)\\
&=(2k-l+1)l+(k+1+l)\frac{k}{l}+(l-1) \frac{\frac{k}{l}(1+\frac{k}{l})}{2}.
\end{align*}
Putting this expression over a common denominator gives the formula for $|\lambda(T)|$ from the statement of the theorem.

We now show that $\A(T) = S$.  The relations in $(\M(S),\preccurlyeq)$ are $a\preccurlyeq 3k+1+b$, for $1\leq b\leq a\leq k-1$. Thus,
\[
I=\{1,2,\dots,l-1\}\cup \{3k+1<x\leq 4k+1\colon x\not\equiv 1\pmod{l}\}
\]
is an order ideal of the void poset. The pseudo-Frobenius numbers in $I$ are
\[
I\cap \PF(S)= \{3k+1<x\leq 4k+1 \colon x\not\equiv 1\pmod{l}\}.
\]
Given $x\in I\cap \PF(S)$, say $1-x\equiv a\pmod{l}$ for $1\leq a\leq l-1$ (since $x\not\equiv 1\pmod{l}$). Note that $a\in I$, and $x+a$ is the smallest positive integer larger than $x$ that is congruent to $1$ modulo $l$.  This implies $3k+1\leq x+a\leq 4k+1$. We see that $x+a\notin I$. 
\begin{itemize}
    \item If $x+a=4k+1$, then $a=\F(S)-x\in I$.
    \item If $x+a<4k+1$, then $x+a\in \M(S)$, so $(x,a, \F(S)-x-a)$ is a Frobenius triangle. Moreover, $a\in I$ and $x+a= \F(S)- (\F(S)-x-a)\notin I$. Therefore, $I$ satisfies the Frobenius triangle $(x,a, \F(S)-x-a)$.
\end{itemize}
By Theorem~\ref{thm:classify_num_set}, it follows that $\A(T)=S$.
\end{proof}

\begin{corollary}
Let $k=12l_1^2$ and $l=3l_1$ in Theorem \ref{thm:2k1_3k1}. Then we have
\[
|\lambda(T)| =2\sqrt{3} |\lambda(S)|^{3/4} +O(|\lambda(S)|^{1/2}).
\]
\end{corollary}
\begin{proof}

In this case $|\lambda(S)|=k^2+4k=144l_1^4+O(l_1^2)$ and $|\lambda(T)|=144l_1^3+O(l_1^2)$. The result follows.
\end{proof}

The following corollary is a more precise version of Theorem \ref{thm:2k1_3k1_intro} from the introduction.
\begin{corollary}
Let $k \ge 4$ and $S = \langle 2k+1, 2k+2,\ldots, 3k+1\rangle$.  

If $l \not\in \{1,k\}$, then the construction of Theorem \ref{thm:2k1_3k1} gives a numerical set $T$ with $\lambda(T) \in \mathcal{P}(S)$ and $|\lambda(T)| < |\lambda(S)|$.  Therefore, if $k$ is not prime, then $S$ is not $\lambda$-minimal.

Let $N$ be a positive integer.  For any $k$ with sufficiently many divisors, $S$ has the property that there are at least $N$ partitions in $\mathcal{P}(S)$ with size less than $|\lambda(S)|$.
\end{corollary}
\begin{proof}
Let $\frac{k}{l} =u$, so $u>1$. Then,
\[
|\lambda(S)|-|\lambda(T)| = (l-1)(u-1)\Big(lu-l-\frac{u}{2}\Big).
\]
Since $l \not\in \{1,k\}$, we have $(l-1)\ge 1$ and $(u-1)\ge 1$.  We see that
\[
lu-l-\frac{u}{2}=\left(l-\frac{1}{2}\right)(u-1)-\frac{1}{2} \geq \frac{3}{2} \cdot 1-\frac{1}{2}=1.
\]

The second part of the statement follows directly from the first.
\end{proof}

\section{The Staircase Family}\label{sec:staircase}
We now turn our attention to a family of numerical semigroups that has received recent interest in relation to the Anti-Atom Problem. 

\begin{definition}[Staircase family of numerical semigroups]
    \label{def:staircase}
    Let $m$ and $k $ be positive integers with $m \ge 2$ and let $s$ be an integer satisfying $1 \le s \le m-1$. A numerical semigroup of the form 
    \[
    S_{m, k, s}=\{0,m,2m,...,km,km+s+1\rightarrow\}
    \]
    is said to be in a \emph{staircase family}. 
\end{definition}
In this section, we will show that every numerical semigroup of the form $S_{m,k,s}$ is $\lambda$-minimal. We begin by describing the pseudo-Frobenius numbers, void, and void poset of a numerical semigroup in the staircase family. When the semigroup $S$ is clear from context, we write $F$ instead of $\F(S)$.  
\begin{lemma}\label{lem:staircase}
Let $S = S_{m,k,s}$ be a numerical semigroup in a staircase family. Then 
\begin{enumerate}
    \item The pseudo-Frobenius elements of $S$ are precisely the elements of the set 
    $\PF(S) = \{(k-1)m+s+1,\ldots, (k-1)m +(m-1), km+1,\ldots, km+s\}.$
    In particular, $\type(S) = m -1$ and $\F(S) = km + s$.
    \item Suppose $x, y \in \M(S)$ with $x \le y$.  Then $x\preccurlyeq y$ in $(\M(S),\preccurlyeq)$ if and only if $m \mid (y-x)$.

    \item The void poset of $S,\ (\M(S),\preccurlyeq)$, consists of $m-2$ chains of elements $x_1\preccurlyeq x_2\preccurlyeq\cdots\preccurlyeq x_n$ where $n \in \{k,k+1\}$. In particular, $x_1, \dots x_n$ are positive elements of a nonzero congruence class modulo $m$ that does not contain $\F(S)$.
\end{enumerate}
\end{lemma}

\begin{proof}
The formula for $\F(S)$ is clear from the description of the elements of $S$.  Furthermore,  $x \in \M(S)$ if and only if $F-x \notin S$. Since $F-x < F$, this is equivalent to saying that $F - x \not\equiv 0 \pmod{m}$. That is, every element of $\G(S)$ not congruent to $F$ modulo $m$ is an element of $\M(S)$. Furthermore, let $x,y\in \M(S)$ with $x \preccurlyeq y$.  We see that $y-x <F$, so $y - x \in S$ if and only if $x \equiv y \pmod{m}$. Considering the set of elements of $\M(S)$ in each residue class modulo $m$, we conclude that each set has either $k$ or $k+1$ elements. Using that fact that pseudo-Frobenius numbers not equal to $F$ are exactly the maximal elements of the void poset, we conclude they are indeed the elements described above. 
\end{proof}

Suppose that $S$ is a numerical semigroup and $T$ is a numerical set associated to $S$. We describe a way to divide the Young diagram of $\lambda(T)$ into sections. Let $m=\m(S)$, $F=\F(S)$, and $q=\q(S)$.
For any $b$ satisfying $0\leq b\leq q-1$, $bm$ is an element of $S$ with $bm<F$. In particular, $bm\in T$. We draw a vertical line in $\lambda(T)$ to the left of the step on the profile of $\lambda(T)$ corresponding to $bm$. 
For any $a$ satisfying $0\leq a\leq q-1$, $F-am\notin S$ and $F-am\notin \M(S)$. Therefore, $F-am\notin T$. We draw a horizontal line in $\lambda(T)$ above the step on the profile of $\lambda(T)$ corresponding to $F-am$.
This divides the Young diagram of $\lambda(T)$ into sections indexed by pairs $(a,b)$ satisfying $0\leq a,b\leq q-1$ and $a+b\leq q-1$. We let
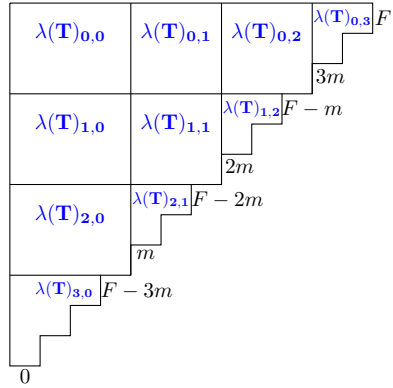
\[
\lambda(T)_{a,b}=\left\{\Box(u,x) : u\in T, x\in \G(T), u<x, \left\lfloor \frac{u}{m}\right\rfloor=b, \left\lfloor\frac{F-x}{m}\right\rfloor=a\right\}.
\]

\begin{center}
\begin{tikzpicture}[scale=0.4, x=1cm,y=-1cm]
\draw (0,0) -- (0,12)--
(1,12) -- (1,11)--
(2,11) -- (2,10)--
(3,10) -- (3,9)--
(4,9) -- (4,8)--
(5,8) -- (5,7)--
(6,7) -- (6,6)--
(7,6) -- (7,5)--
(8,5) -- (8,4)--
(9,4) -- (9,3)--
(10,3) -- (10,2)--
(11,2) -- (11,1)--
(12,1) -- (12,0)--
(0,0);

\draw (0,3) -- (9,3);
\draw (0,6) -- (6,6);
\draw (0,9) -- (3,9);

\draw (4,0) -- (4,9);
\draw (7,0) -- (7,6);
\draw (10,0) -- (10,3);

\node[scale=0.8,color=blue] at (2,1) {$\mathbf{\lambda(T)_{0,0}}$};
\node[scale=0.8,color=blue] at (5.5,1) {$\mathbf{\lambda(T)_{0,1}}$};
\node[scale=0.8,color=blue] at (8.5,1) {$\mathbf{\lambda(T)_{0,2}}$};
\node[scale=0.65,color=blue] at (11,0.5) {$\mathbf{\lambda(T)_{0,3}}$};

\node[scale=0.8,color=blue] at (2,4) {$\mathbf{\lambda(T)_{1,0}}$};
\node[scale=0.8,color=blue] at (5.5,4) {$\mathbf{\lambda(T)_{1,1}}$};
\node[scale=0.65,color=blue] at (8,3.5) {$\mathbf{\lambda(T)_{1,2}}$};

\node[scale=0.8,color=blue] at (2,7) {$\mathbf{\lambda(T)_{2,0}}$};
\node[scale=0.65,color=blue] at (5,6.5) {$\mathbf{\lambda(T)_{2,1}}$};

\node[scale=0.65,color=blue] at (1.8,9.5) {$\mathbf{\lambda(T)_{3,0}}$};

\node[scale=0.8] at (0.5,12.35) {$0$};
\node[scale=0.8] at (4.5,8.3) {$m$};
\node[scale=0.8] at (7.6,5.4) {$2m$};
\node[scale=0.8] at (10.6,2.4) {$3m$};

\node[scale=0.8] at (12.35,0.5) {$F$};
\node[scale=0.8] at (10,3.5) {$F-m$};
\node[scale=0.8] at (7.2,6.5) {$F-2m$};
\node[scale=0.8] at (4.2,9.5) {$F-3m$};
\end{tikzpicture}
\end{center}

We denote the set of hook lengths that appear in the boxes of $\lambda(T)_{a,b}$ by
\[
\mathcal{H}(T)_{a,b}=\{|\hook(u,x)| : \Box(u,x)\in \lambda(T)_{a,b}\}.
\]
It is clear that $|\mathcal{H}(T)_{a,b}| \leq |\lambda(T)_{a,b}|$.

We can now give the proof of Theorem \ref{thm:staircase}.
We want to show that if $T$ is a numerical set associated to $S=S_{m,k,s}$, then $|\lambda(S)|\leq |\lambda(T)|$.
We do this by showing that for each section we have $|\lambda(S)_{a,b}| \leq |\mathcal{H}(T)_{a,b}|$.

\begin{proof}[Proof of Theorem \ref{thm:staircase}]
Let $S=S_{m,k,s}$, so $\m(S)=m$, $F = \F(S)=km+s$, and $\q(S)=k+1$. Note that each section of the Young diagram of $\lambda(S)$ is contained in a single column. We see that
    \[
|\lambda(S)_{a,b}| = 
\begin{cases}
m-1 & \text{ if } a+b < k,\\
s & \text{ if } a+b = k.
\end{cases}
    \]
    \begin{center}
    \begin{figure}
    \begin{tikzpicture}[scale=0.75, x=1cm, y=0.65cm]
        \coordinate (A) at (0,0);
        \coordinate (B) at (0,3/2);
        \coordinate (C) at (0,7/2);
        \coordinate (D) at (0,15/2);
        \coordinate (E) at (0,9.5);
        \coordinate (F) at (1/2,0);
        \coordinate (G) at (1/2,3/2);
        \coordinate (H) at (1/2,2);
        \coordinate (I) at (1,2);
        \coordinate (J) at (1,7/2);
        \coordinate (K) at (1,4);
        \coordinate (L) at (3/2,4);
        \coordinate (M) at (2,6);
        \coordinate (N) at (2,15/2);
        \coordinate (O) at (2,8);
        \coordinate (P) at (5/2,8);
        \coordinate (Q) at (5/2,9.5);
        \coordinate (R) at (1/2,9.5);
        \coordinate (S) at (1,9.5);
        \coordinate (T) at (2,9.5);
        \coordinate (U) at (0,4);
        \coordinate (V) at (0,6);
        \coordinate (W) at (1/2,4);
        \coordinate (X) at (1/2,6);
        \coordinate (Y) at (1,6);
        \coordinate (Z) at (1.5,6);
        \coordinate (z) at (1.5,9.5);
        \coordinate (aa) at (1/2,4);
        \coordinate (bb) at (1/2,6);

        \draw (A)--(U);
        \draw[dashed] (U)--(V);
        \draw (V)--(E);
        \draw (A)--(F);
        \draw (F)--(W);
        \draw (H)--(I);
        \draw (B)--(G);
        \draw (I)--(K);
        \draw (K)--(L);
        \draw (C)--(J);
        \draw[dashed] (L)--(M);
        \draw (M)--(O);
        \draw (N)--(D);
        \draw (O)--(P);
        \draw (P)--(Q);
        \draw (E)--(Q);
        \draw (X)--(R);
        \draw[dashed] (aa)--(bb);
        \draw (Y)--(S);
        \draw[dashed] (Y)--(K);
        \draw (O)--(T);
        \draw[dashed] (L)--(Z);
        \draw (Z)--(z);

        \node [scale=3/4] at (1/4,-0.25) {$0$};
        \node [scale=3/4] at (1.15,1.2) {$F-km$};
        \node [scale=3/4] at (0.78,1.77) {$m$};
        \node [scale=0.7] at (2.05,3.1) {$F-(k-1)m$};
        \node [scale=3/4] at (1.3,3.75) {$2m$};
        \node [scale=3/4] at (2.55,7.1) {$F-m$};
        \node [scale=3/4] at (2.3,7.75) {$km$};
        \node [scale=3/4] at (3.1,9.2) {$km+s$};
    \end{tikzpicture}
\caption{Young Diagram of $S_{m,k,s}$}
\end{figure}
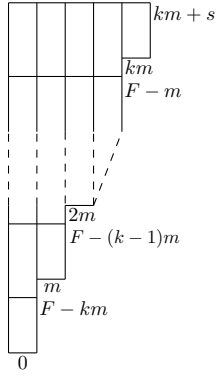
    \end{center}

Let $T$ be a numerical set associated to $S$.
Our goal is to show that in each section $|\lambda(T)_{a,b}| \ge |\lambda(S)_{a,b}|$. Since $|\lambda(T)_{a,b}| \ge |\mathcal{H}(T)_{a,b}|$, it is enough to show that $|\mathcal{H}(T)_{a,b}| \ge |\lambda(S)_{a,b}|$.

In the case where $a+b<k$, we want to show that $|\mathcal{H}(T)_{a,b}|\geq m-1$. We will do this by showing that for each $\ell\in [1,m-1]$, $\mathcal{H}(T)_{a,b}$ contains a hook length congruent to $\ell$ modulo $m$.
Note that $\lambda(T)_{a,b}$ contains $\Box(u,x)$ for
\begin{align*}
u & \in [bm,(b+1)m-1]\cap T,
&
x & \in [F-am-(m-1),F-am]\cap \G(T).
\end{align*}

In the case where $a+b=k$, we want to show that $|\mathcal{H}(T)_{a,b}|\geq s$. We will do this by showing that for each $\ell \in [1,s],\ \lambda(T)_{a,b}$ contains a box with hook length congruent to $\ell$ modulo $m$.
Since $F-am=bm+s$, we have that $\lambda(T)_{a,b}$ contains $\Box(u,x)$, where $u<x$ and
\begin{align*}
u & \in  [bm,bm+s]\cap T,
&
x & \in [bm,bm+s]\cap \G(T).
\end{align*}

    Consider $\ell\in [1,m-1]$. We will show that $\lambda(T)_{a,b}$ contains a box with hook length congruent to $\ell$ modulo $m$. If $a+b=k$, then we include the additional requirement that $\ell \leq s$.
    Say $F-\ell\equiv \alpha \pmod{m}$ with $0\leq \alpha\leq m-1$ and $x_1=F-am-\alpha$. So $x_1\equiv \ell \pmod{m}$ and $x_1\in [F-am-(m-1),F-am]$.
    
    First, consider the case where $x_1\notin T$. The box $\Box(bm,x_1)$ has hook length congruent to $\ell$ modulo $m$.
    If $a+b<k$, then this box is clearly in $\lambda(T)_{a,b}$.  If $a+b=k$, then $\ell\leq s$, so $\alpha\leq s$. Therefore, $x_1\in [bm,bm+s]$ and we see that $\Box(bm,x_1)$ is in $\lambda(T)_{a,b}$.
    
    Next, consider the case where $x_1\in T$. Suppose $T = S \cup I$ for an order ideal $I$. Since $x_1 \equiv \ell \not\equiv 0 \pmod{m}$, we know that $x_1\notin S$. This means
    $x_1\in I\subseteq \M(S)$, so $x_1\not\equiv F\pmod{m}$, and hence $\alpha\neq 0$.
    Since $I$ is an order ideal, the description of $(\M(S),\preccurlyeq)$ from Lemma~\ref {lem:staircase} implies that $I$ contains the pseudo-Frobenius number of $S$ that is congruent to $x_1$ modulo $m$. We call this pseudo-Frobenius number $P$.  Then by Theorem \ref{thm:classify_num_set} we have two cases.
        \begin{enumerate}[leftmargin=*]
            \item Suppose $F-P\in T$. Proposition \ref{prop:max_min_elements} implies that $F-P$ is a minimal element of $\M(S)$ and Lemma \ref{lem:staircase} implies that $T$ contains all positive integers congruent to $F-P$ modulo $m$. Note that $bm+\alpha\equiv F-P\pmod{m}$ and the box $\Box(bm+\alpha,F-am)$ has hook length congruent to $\ell$ modulo $m$. If $a+b<k$, then this box is clearly in $\lambda(T)_{a,b}$. In the case where $a+b=k$, we have $0<\ell \leq s$, which implies $0\leq \alpha< s$, so $bm+\alpha\in [bm,bm+s]$. Moreover, $F-am=bm+s$, so $F-am\in [bm,bm+s]$ and $bm+\alpha<F-am$. Thus, $\Box(bm+\alpha,F-am)$ is in $\lambda(T)_{a,b}$.
            
            \item Suppose $I$ satisfies a Frobenius triangle $(P,y,z)$. We have $y\in T$, $F-z\notin T$, and $P+y+z=F$. Lemma \ref{lem:staircase} implies that $F-P<m$. Therefore, $y+z=F-P <m$. Since $y,z\in \M(S)$, Proposition \ref{prop:max_min_elements} implies that both $y$ and $z$ are minimal elements of $(\M(S),\preccurlyeq)$. Since $y\in I$ and $I$ is an order ideal, Lemma~\ref{lem:staircase} implies that $T$ contains all positive integers congruent to $y$ modulo $m$. Since $z$ is a minimal element of $(\M(S),\preccurlyeq)$, Proposition \ref{prop:max_min_elements} implies that $F-z$ is a maximal element of $(\M(S),\preccurlyeq)$. Since $F-z\notin I$ and $I$ is an order ideal, Lemma~\ref{lem:staircase} implies that $T$ does not contain any positive integer less than $\F(S)$ that is congruent to $F-z$ modulo $m$.
            
            The box $\Box(bm+y,F-am-z)$ has hook length $F-am-z-bm-y\equiv P\equiv \ell \pmod{m}$.
            If $a+b<k$, then this box is clearly in $\lambda(T)_{a,b}$. In the case where $a+b=k$, we know that $1 \le \ell \le s$, which implies $F-P\leq s$. Also, we are in the case $x_1\in T$, so $x_1\not\equiv F\pmod{m}$, that is, $\ell \neq s$. Recall that $F-am=bm+s$. 
            We see that $y+z=F-P< s$, so $bm+y,F-am-z\in [bm,bm+s]$. Moreover, $bm+y<F-am-z$ as $y+z<s$. Hence, $\Box(bm+y,F-am-z)$ is in $\lambda(T)_{a,b}$.
        \end{enumerate}
For sections with $a+b<k$, we have shown that 
\[
|\lambda(T)_{a,b}|\geq |\mathcal{H}(T)_{a,b}|\geq m-1=|\lambda(S)_{a,b}|.
\]
For sections with $a+b=k$, we have shown that 
\[
|\lambda(T)_{a,b}|\geq |\mathcal{H}(T)_{a,b}|\geq s=|\lambda(S)_{a,b}|.
\]
This completes the proof.
\end{proof}

\section{Families of Numerical Semigroups with Depth $2$}\label{sec:depth2}

In this section, we consider several families of numerical semigroups with depth~$2$ and show that they are $\lambda$-minimal.

\subsection{Depth 2 and Durfee Squares}

Our goal is to analyze numerical semigroups of depth 2 with a Durfee square of fixed size. 
We provide a sufficient condition on the large gaps of such numerical semigroups to ensure that they are $\lambda$-minimal.
We then show that if $S$ has depth $2$ and the Durfee square of $\lambda(S)$ has size at most $3$, this condition is always satisfied, which allows us to conclude that all such numerical semigroups $S$ are $\lambda$-minimal.

If the Durfee square of $\lambda(S)$ has size $n$, we will partition the Young diagram of $\lambda(S)$ into a union of $n$ hooks. If $T$ is any numerical set associated to $S$, then we find $n$ disjoint hooks of corresponding lengths in the Young diagram of $\lambda(T)$. 

Recall that each hook is uniquely determined by a pair $(a,x)$ satisfying $a\in T$, $x\in \G(T)$, and $a<x$. Moreover, two hooks $\hook(a_1,x_1)$ and $\hook(a_2,x_2)$ with $a_1\leq a_2$ are disjoint if either $a_1<a_2<x_2<x_1$ or $a_1<x_1<a_2<x_2$.

\begin{theorem}
    \label{general durfee}
    Let $S$ be a numerical semigroup with depth 2 for which $\lambda(S)$ has a Durfee square of size $n$. Let $F = \F(S)$ and suppose that $F, F-\alpha_1, F-\alpha_2, F-\alpha_3,...,F-\alpha_{n-1}$ are the $n$ largest gaps of $S$ in descending order.
    
Suppose $\alpha_1,\ldots, \alpha_{n-1}$ have the property that $\alpha_i=\alpha_j+\alpha_k$ implies $j=k=i-1$.
    Then $S$ is $\lambda$-minimal.
\end{theorem}
\begin{proof}
Since the Durfee square of $\lambda(S)$ has size $n$, the Young diagram of $\lambda(S)$ can be partitioned as a disjoint union $n$ hooks. Suppose that the $n$ smallest elements of $S$ are $0<F-\beta_1<\dots<F-\beta_{n-1}$. Note that $\m(S)=F-\beta_1$. Therefore, $\lambda(S)$ is the disjoint union of the hooks 
\[
\hook(0,F),\ \hook(F-\beta_1,F-\alpha_1),\ldots, \hook(F-\beta_{n-1},F-\alpha_{n-1}).
\]
These hooks have sizes $F, \beta_1-\alpha_1,\dots, \beta_{n-1}-\alpha_{n-1}$, and therefore 
\[
|\lambda(S)|=F+\sum_{i=1}^{n-1}(\beta_i-\alpha_i).
\]
Note that $0<\alpha_1<\dots<\alpha_{n-1}<\beta_{n-1}<\dots<\beta_1<\frac{F}{2}$, where the last inequality follows from the assumption that $S$ has depth 2.
The following diagram illustrates these hooks.
\begin{center}
\begin{figure}
\begin{tikzpicture}[scale=0.4, x=1cm,y=-1cm]
  \draw[thick] (0,0) rectangle (5,5);
  \foreach \i in {1,2,3,4}{
    \draw (0,\i) -- (5,\i);
    \draw (\i,0) -- (\i,5);
  }

  \draw (5,0) -- (12,0);  
  \draw (5,1) -- (10,1);  
  \draw (5,2) -- (9,2);  
  \draw (5,3) -- (7,3);
  \draw (5,4) -- (6,4);

  \draw (0,5) -- (0,10);  
  \draw (1,5) -- (1,9);  
  \draw (2,5) -- (2,8);  
  \draw (3,5) -- (3,7);
  \draw (4,5) -- (4,6);

  \draw[thick]
    (0,0) --
    (12,0) --(12,1) --
    (10,1) -- (10,2) --
    (9,2) -- (9,3) --
    (7,3) -- (7,4) --
    (6,4) -- (6,5) --
    (6,5) -- (5,5) --
    (5,6) -- (4,6) --
    (4,7) -- (3,7) --
    (3,8) -- (2,8) --
    (2,9) -- (1,9) --
    (1,10) -- (0,10) --
    (0,0);

  \node[right] at (12,0.5) {$F$};
  \node[right] at (10,1.5) {$F - \alpha_1$};
  \node[right] at (9,2.55) {$F - \alpha_2$};
  \node[right] at (6,4.6) {$F - \alpha_{n-1}$};

  \node[below,scale=0.65] at (0.75,10) {$0$};
  \node[below,scale=0.65] at (1.95,9) {$F-\beta_1$};
  \node[below,scale=0.65] at (2.95,8) {$F - \beta_2$};
  \node[below,scale=0.65] at (5.25,6) {$F - \beta_{n-1}$};

  \draw[red,very thick]
    (0.5,0.5) -- (12,0.5);
\draw[red,very thick]
    (0.5,0.5) -- (0.5,10);
  \draw[blue,very thick]
    (1.5,1.5) -- (10,1.5);
\draw[blue,very thick]
    (1.5,1.5) -- (1.5,9);
  \draw[green!50!black,very thick]
    (2.5,2.5) -- (9,2.5);
\draw[green!50!black,very thick]
    (2.5,2.5) -- (2.5,8);
    \draw[pink,very thick]
    (4.5,4.5) -- (6,4.5);
\draw[pink,very thick]
    (4.5,4.5) -- (4.5,6);
\end{tikzpicture}
\caption{Decomposition of Young diagram into $n$ hooks}
\end{figure}
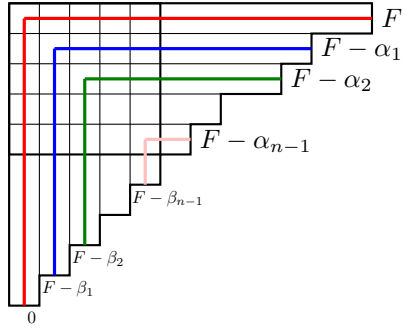
\end{center}

We now consider the set of ideal triangles of $S$ that could potentially contain $F-\alpha_i$. Suppose $(F-\alpha_i,u,v)$ is an ideal triangle of $S$. Then $F-v$ is a gap of $S$ and $F>F-v=(F-\alpha_i)+u>F-\alpha_i$. This implies $F-v=F-\alpha_{j_1}$ for some $j_1$. Similarly $F-u=F-\alpha_{j_2}$, hence $v=\alpha_{j_1}$ and $u=\alpha_{j_2}$. If $(F-\alpha_i,\alpha_{j_1},\alpha_{j_2})$ is an ideal triangle,  then $\alpha_i=\alpha_{j_1}+\alpha_{j_2}$. The assumption in the statement of the theorem implies $j_1=j_2=i-1$.

Let $T$ be a numerical set associated to $S$.
We will find $n$ disjoint hooks in $\lambda(T)$, of sizes at least $F, \beta_1-\alpha_1,\dots,\beta_{n-1}-\alpha_{n-1}$. We choose the first hook to correspond to the pair $(0,F)$. This hook has size $F$.
None of $F-\alpha_1,\ldots, F-\alpha_{n-1}$ are elements of $S$. Proposition~\ref{Element in I} implies that for each $i$ satisfying $1\le i \le n-1$ one of the following holds 
\begin{itemize}
\item $F-\alpha_i\notin T$;
\item $F-\alpha_i,\alpha_i\in T$;
\item $F- \alpha_i, \alpha_{i-1} \in T$, $2\alpha_{i-1} = \alpha_i$, and $F-\alpha_{i-1} \not\in T$.
\end{itemize}

In the first case, we choose the $(i+1)$\textsuperscript{st} hook in our list to be $\hook(F-\beta_{i},F-\alpha_i)$. Since $F-\beta_i\in S\subseteq T$ and $F-\beta_i<F-\alpha_i$, we see that $\hook(F-\beta_{i},F-\alpha_i)$ is a hook in the Young diagram of $\lambda(T)$. This hook has size $\beta_i-\alpha_i$.

In the second case, we choose the $(i+1)$\textsuperscript{st} hook in our list to be $\hook(\alpha_i,\beta_i)$. Since $F-\beta_i \in S$, we see that $\beta_i$ is a gap of $S$ that is not in $\M(S)$.  Therefore, $\beta_i \not\in T$.
Moreover, $\alpha_i<\beta_i$, so this is indeed a valid hook. It has size $\beta_i-\alpha_i$.

If the third condition holds and the second does not, then we choose the $(i+1)$\textsuperscript{st} hook in our list to be $\hook(\alpha_{i-1},\beta_i)$. As in the previous case, we see that $\beta_i \not\in T$. 
Since $\alpha_{i-1}<\alpha_i<\beta_i$, we see that $\hook(\alpha_{i-1},\beta_i)$ is a valid hook in the Young diagram of $\lambda(T)$. This hook has size $\beta_i-\alpha_{i-1}>\beta_i-\alpha_i$.

Consider distinct $1\leq i_1,i_2\leq n-1$. We will show that the hooks in our list corresponding to $i_1$ and $i_2$ are disjoint.
\begin{enumerate}
    \item Suppose that the hooks are $\hook(F-\beta_{i_1},F-\alpha_{i_1})$, $\hook(F-\beta_{i_2},F-\alpha_{i_2})$. Assume $i_1<i_2$. Then we have $F-\beta_{i_1}<F-\beta_{i_2}<F-\alpha_{i_2}<F-\alpha_{i_1}$.
    \item Suppose that the hooks are $\hook(F-\beta_{i_1},F-\alpha_{i_1})$, $\hook(\alpha_{i_2},\beta_{i_2})$. Then we have $\alpha_{i_2}<\beta_{i_2}< F-\beta_{i_1}<F-\alpha_{i_1}$.
    \item Suppose that the hooks are $\hook(\alpha_{i_1},\beta_{i_1})$, $\hook(\alpha_{i_2},\beta_{i_2})$. Without loss of generality, we may assume that $i_1<i_2$. We see that $\alpha_{i_1}<\alpha_{i_2}<\beta_{i_2}<\beta_{i_1}$.
    \item Suppose that the hooks are $\hook(F-\beta_{i_1},F-\alpha_{i_1})$, $\hook(\alpha_{i_2-1},\beta_{i_2})$.
    Then we have $\alpha_{i_2-1}<\beta_{i_2}< F-\beta_{i_1}<F-\alpha_{i_1}$.
    \item Suppose that the hooks are $\hook(\alpha_{i_1},\beta_{i_1})$, $\hook(\alpha_{i_2-1},\beta_{i_2})$.
    For this case to happen, it is necessary that $F-\alpha_{i_1}\in T$ and $F-\alpha_{i_2-1}\notin T$.  This implies $i_1\neq i_2-1$.
    
    We first consider the sub-case where $i_2<i_1$. Then we have $\alpha_{i_2-1}<\alpha_{i_1}<\beta_{i_1}<\beta_{i_2}$.

    Next, consider the sub-case where $i_2>i_1$. Since $i_2-i_1\neq 1$, this implies $i_2-1>i_1$. In this case, we have $\alpha_{i_1}<\alpha_{i_2-1}<\beta_{i_2}<\beta_{i_1}$.

    \item Suppose that the hooks are $\hook(\alpha_{i_1-1},\beta_{i_1})$, $\hook(\alpha_{i_2-1},\beta_{i_2})$.
    Without loss of generality, we may assume that $i_1<i_2$. We see that $\alpha_{i_1-1}<\alpha_{i_2-1}<\beta_{i_2}<\beta_{i_1}$.
\end{enumerate}
Therefore, in all cases, the chosen hooks are disjoint from each other.
This shows that 
\[|\lambda(T)|\geq F+\sum\limits_{i=1}^{n-1}(\beta_i-\alpha_i)=|\lambda(S)|.\]
\end{proof}

\begin{proof}[Proof of Theorem~\ref{thm:Durfee}]
If $n\leq 3$, then $\alpha_1,\dots,\alpha_{n-1}$ always satisfy the condition that $\alpha_i=\alpha_j+\alpha_k$ implies $j=k=i-1$. Applying Theorem~\ref{general durfee} completes the proof.
\end{proof}

If we consider collections $\alpha_1,\ldots, \alpha_{n-1}$ for which $\alpha_i=\alpha_j+\alpha_k$ does not necessarily imply $j=k=i-1$, then the same arguments do not work. Let $S$ be one of the numerical semigroups from the statement of Theorem~\ref{thm:inf_non_lam_sizes}. The depth of $S$ is $2$ and $\lambda(S)$ has a Durfee square of size $4$.  The four largest gaps of $S$ are
\[
2m-1,2m-2,2m-3,2m-4, 
\]
with $\F(S) = 2m-1$. So $(\alpha_1, \alpha_2, \alpha_3)=(1,2,3)$. We have that $\alpha_3=\alpha_1+\alpha_2$. Therefore, Theorem~\ref{general durfee} does not apply, and as was shown in Theorem~\ref{thm:inf_non_lam_sizes}, $S$ is not $\lambda$-minimal. This example also shows that Theorem~\ref{thm:Durfee} is sharp in the sense that there are infinitely many numerical semigroups $S$ of depth $2$ for which $\lambda(S)$ has a Durfee square of size $4$ that are not $\lambda$-minimal.

\subsection{Depth 2 and Few Small Elements}

Our work in this section is motivated by some results of Singhal and Lin \cite{doi:10.1080/00927872.2021.1918136} that build on earlier work of Marzuola and Miller \cite{MarzuolaMiller}. They show that when considering all numerical sets $T$ with a fixed Frobenius number, $\A(T)$ is very likely to be of depth $2$ with $0,1,2$, or $3$ small elements. We show that in each of these cases, $\A(T)$ is $\lambda$-minimal. 

We begin by noting that the number of small elements and the Durfee square size of a numerical semigroup are closely related. 
\begin{lemma}\label{small to durfee lem}
    If a numerical semigroup $S$ has $k$ small elements, then the size of the Durfee square of $\lambda(S)$ is at most $k+1$. 
\end{lemma}
This follows directly from noting that the number of columns of $\lambda(S)$ is one more than the number of small elements of $S$. 
We demonstrate this with Example~\ref{small ele to durfee}.
\begin{example}
    \label{small ele to durfee}
    Consider the numerical semigroup $S=\{0,5,6,10,\rightarrow\}$. The Young diagram of $\lambda(S)$ is shown in Figure 4.
    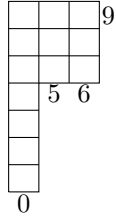
\begin{figure}[H]\label{fig:Young}
    \begin{tikzpicture}[scale=0.4,x=1cm,y=0.9cm]
        \coordinate (A) at (0,0);
        \coordinate (B) at (0,1);
        \coordinate (C) at (0,2);
        \coordinate (D) at (0,3);
        \coordinate (E) at (0,4);
        \coordinate (F) at (0,5);
        \coordinate (G) at (0,6);
        \coordinate (H) at (0,7);

        \coordinate (I) at (1,0);
        \coordinate (J) at (1,1);
        \coordinate (K) at (1,2);
        \coordinate (L) at (1,3);
        \coordinate (M) at (1,4);

        \coordinate (N) at (2,4);

        \coordinate (O) at (3,4);
        \coordinate (P) at (3,5);
        \coordinate (Q) at (3,6);
        \coordinate (R) at (3,7);

        \coordinate (S) at (1,7);
        \coordinate (T) at (2,7);

        \draw (A)--(H)--(R)--(O)--(M)--(I)--(A);
        \draw (B)--(J);
        \draw (C)--(K);
        \draw (D)--(L);
        \draw (E)--(M);
        \draw (F)--(P);
        \draw (G)--(Q);
        \draw (M)--(S);
        \draw (N)--(T);

        \node at (0.5,-0.44) {0};
        \node at (1.5,3.6) {5};
        \node at (2.5,3.6) {6};
        \node at (3.3,6.5) {9};

    \end{tikzpicture}
\caption{The Young diagram of $\lambda(S)$}
    \end{figure}
\noindent The small elements of $S$ are $5$ and $6$, and $\lambda(S)$ has a Durfee square of size $3$. 
\end{example}

\begin{corollary}
    Numerical semigroups with depth $2$ and at most $2$ small elements are $\lambda$-minimal.
\end{corollary}

\begin{proof}
    If $S$ has at most $2$ small elements, then Lemma \ref{small to durfee lem} implies that $\lambda(S)$ has a Durfee square of size at most $3$. Since $S$ has depth $2$, Theorem \ref{thm:Durfee} implies that $S$ is $\lambda$-minimal.
\end{proof}

In this section, we only consider numerical semigroups of depth $2$. This significantly simplifies the notation for the sections of the Young diagram of $\lambda(T)$ that we introduced in Section \ref{sec:staircase}. Let $\lambda(T)_{1}=\lambda(T)_{1,0}$, $\lambda(T)_{2}=\lambda(T)_{0,0}$, and $\lambda(T)_{3}=\lambda(T)_{0,1}$.
We also denote $\mathcal{H}(T)_{1}=\mathcal{H}(T)_{1,0}$, $\mathcal{H}(T)_{2}=\mathcal{H}(T)_{0,0}$, and $\mathcal{H}(T)_{3}=\mathcal{H}(T)_{0,1}$.

\begin{lemma}\label{lem: section 2}
If $S$ is a numerical semigroup of depth $2$ and $T$ is a numerical set associated to $S$, then
\[
\mathcal{H}(\lambda(T))\cap [\F(S)-\m(S)+1,\F(S)]\subseteq \mathcal{H}(T)_2.
\]
Moreover, if $T=S$, then
\[
\mathcal{H}(\lambda(S))\cap [\F(S)-\m(S)+1,\F(S)]= \mathcal{H}(S)_2
\]
and $|\mathcal{H}(S)_2|=|\lambda(S)_2|$.
\end{lemma}
\begin{proof}
Let $m=\m(S)$ and $F=\F(S)$.
Recall that $\lambda(T)_{1}$ consists of the $\Box(u,x)$ for which $0\leq u<m$ and $x\leq F-m$. This means that every hook length in $\mathcal{H}(T)_1$ is at most $F-m$.
Similarly, $\lambda(T)_{3}$ consists of the $\Box(u,x)$ for which $m\leq u$ and $F-m<x\leq F$. This means that every hook length in $\mathcal{H}(T)_3$ is at most $F-m$.
Therefore, hook lengths larger than $F-m$ can only occur in $\mathcal{H}(T)_{2}$. The first part of the lemma follows.

Next, suppose that $T=S$. Note that $[0,m-1]\cap S=\{0\}$. This means that $\lambda(S)_2$ consists of $\Box(0,x)$ for $x\in \G(S)$ with $F-m<x\leq F$. This shows that every hook length in $\mathcal{H}(S)_2$ is at least $F-m+1$. It follows that $\mathcal{H}(S)_{2}$ has precisely those hook lengths that are larger than $F-m$.
Finally, from the description 
\[
\lambda(S)_2=\{\Box(0,x): x\in \G(S), F-m<x\},
\]
we see that there are no repeated hook lengths in $\lambda(S)_2$. This implies that $|\mathcal{H}(S)_2|=|\lambda(S)_2|$.
\end{proof}

\begin{corollary}\label{cor: section 2}
If $S$ is a numerical semigroup of depth $2$ and $T$ is a numerical set associated to $S$, then $|\lambda(S)_2|\leq |\lambda(T)_2|$.
\end{corollary}
\begin{proof}
Since $\A(T)=S$, we know that $\mathcal{H}(\lambda(T))=\mathcal{H}(\lambda(S))$. Therefore, Lemma~\ref{lem: section 2} implies that
\[\mathcal{H}(S)_2= \mathcal{H}(\lambda(S)) \cap [F-m+1,F] = \mathcal{H}(\lambda(T)) \cap [F-m+1,F] \subseteq \mathcal{H}(T)_2.\]
This shows that $|\lambda(S)_2|= |\mathcal{H}(S)_2|\leq |\mathcal{H}(T)_2|\leq |\lambda(T)_2|$.
\end{proof}

\begin{theorem}
    If $S$ is a numerical semigroup of depth $2$ with $3$ small elements, then $S$ is $\lambda$-minimal.
\end{theorem}

\begin{proof}
Let $m = \m(S)$.  Define $a$ and $b$ so that the three small elements of $S$ are $m <F-a<F-b$. So $b<a<F-m<\frac{F}{2}$.
If $b\leq 3$, then the size of the Durfee square of $\lambda(S)$ is at most $3$.  Applying Theorem~\ref{thm:Durfee} completes the proof in this case. Therefore, assume $b\geq 4$. 
Let $T$ be a numerical set associated to $S$.
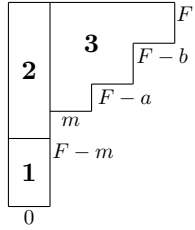
\begin{figure}[H]
\begin{tikzpicture}[x=1.1cm, y=0.9cm]
    \coordinate (A) at (0,0);
    \coordinate (B) at (0,1);
    \coordinate (C) at (0,3);
    \coordinate (D) at (0.5,0);
    \coordinate (E) at (0.5,1);
    \coordinate (F) at (0.5,1.4);
    \coordinate (G) at (0.5,3);
    \coordinate (H) at (1,1.4);
    \coordinate (I) at (1,1.8);
    \coordinate (J) at (1.5,1.8);
    \coordinate (K) at (1.5,2.4);
    \coordinate (L) at (2,2.4);
    \coordinate (M) at (2,3);

    \draw (A)--(C);
    \draw (A)--(D);
    \draw (D)--(G);
    \draw (B)--(E);
    \draw (F)--(H);
    \draw (H)--(I);
    \draw (I)--(J);
    \draw (J)--(K);
    \draw (K)--(L);
    \draw (L)--(M);
    \draw (C)--(M);

    \node[scale=0.8] at (0.25,-0.15) {$0$};
    \node[scale=0.8] at (0.9,0.8) {$F-m$};
    \node[scale=0.8] at (0.75,1.25) {$m$};
    \node[scale=0.8] at (1.4,1.6) {$F-a$};
    \node[scale=0.8] at (1.85,2.2) {$F-b$};
    \node[scale=0.8] at (2.12,2.84) {$F$};

    \node at (0.25,0.5) {\textbf{1}};
    \node at (0.25,2) {\textbf{2}};
    \node at (1,2.4) {\textbf{3}};
    \end{tikzpicture}
\caption{Young diagram of $S=\{0,m,F-a,F-b,F+1\rightarrow\}$}
\end{figure}

By Corollary~\ref{cor: section 2}, we know that 
$|\lambda(S)_2|\leq |\lambda(T)_2|$.
Therefore, it suffices to show that 
    \[
|\lambda(S)_1|+|\lambda(S)_3|\leq|\lambda(T)_1|+|\lambda(T)_3|.
\]
Note that $|\lambda(S)_1|=F-m$. Since $b\geq 3$, we know that $F-1,F-2\in \G(S)$. Therefore, $|\lambda(S)_3|=(F-m)+(a-1)+(b-2)$.
We will construct $4$ disjoint hooks in $\lambda(T)_1\cup \lambda(T)_3$ of sizes at least $F-m$, $F-m$, $a-1$ and $b-2$.

Notice that there are no ideal triangles of $S$ containing $F-1$ and the only possibility for an ideal triangle of $S$ that contains $F-2$ is $(F-2,1,1)$.

First, we know that $0,m\in T$ and $F-m,F\notin T$, so we have hooks $\hook(0,F-m)\subseteq \lambda(T)_1$ and $\hook(m,F)\subseteq \lambda(T)_3$. They both have size $F-m$.

Since $F-1\notin S$, and $F-1$ is not contained in any ideal triangle of $S$, Theorem~\ref{Element in I} implies that we have two possibilities:
\begin{itemize}
    \item $F-1\notin T$. In this case, we take the hook $\hook(F-a,F-1)\subseteq\lambda(T)_3$. Note that this is disjoint from $\hook(m,F)$ as $m<F-a<F-1<F$.
    \item $F-1,1\in T$. In this case, we take the hook $\hook(1,a)\subseteq\lambda(T)_1$. Note that $a\notin T$ as $a\notin S\cup \M(S)$. Therefore, this is a valid hook. Moreover, this hook is disjoint from $\hook(0,F-m)$ as $0<1<a<F-m$.  In either case, we have constructed a hook of size $a-1$.
\end{itemize}

Since $F-2\notin S$ and the only possibility for an ideal triangle of $S$ containing $F-2$ is $(F-2,1,1)$, Theorem~\ref{Element in I} implies that we have three possibilities:
\begin{itemize}
    \item $F-2\notin T$. In this case, we take the hook $\hook(F-b,F-2)\subseteq\lambda(T)_3$. Note that this is disjoint from $\hook(m,F)$ as $m<F-b<F-2<F$.
    It is also disjoint from $\hook(F-a,F-1)$ or $\hook(1,a)$, whichever was chosen.
    \item $F-2,2\in T$. In this case, we take the hook $\hook(2,b)\subseteq\lambda(T)_1$. Note that $b\notin T$ as $b\notin S\cup \M(S)$. Therefore, this is a valid hook. Moreover, this is disjoint from $\hook(0,F-m)$ as $0<2<b<F-m$.
    It is also disjoint from $\hook(F-a,F-1)$ or $\hook(1,a)$, whichever was chosen.
    \item $F-2,1\in T$ and $F-1\notin T$.
    In this case, we take the hook $\hook(1,b)\subseteq\lambda(T)_1$. This is disjoint from $\hook(0,F-m)$ as $0<1<b<F-m$.
    Since $F-1\notin T$, the previous hook chosen must have been $\hook(F-a,F-1)$ which is disjoint from this hook.  In any of these cases, we have constructed a hook of size $b-2$.
\end{itemize}

In all cases we constructed $4$ disjoint hooks in $\lambda(T)_1\cup \lambda(T)_3$ of sizes at least $F-m$, $F-m$, $a-1$ and $b-2$. This implies that
\[
|\lambda(S)_1|+|\lambda(S)_3|=(F-m)+(F-m)+(a-1)+(b-2)\leq |\lambda(T)_1|+|\lambda(T)_3|.
\]
We conclude that $|\lambda(S)|\leq |\lambda(T)|$ and $S$ is $\lambda$-minimal.
\end{proof}

\section{Numerical semigroups of small type}\label{sec:small_type}

We begin this section by showing how earlier results imply that numerical semigroups of type at most $2$ are $\lambda$-minimal.
\begin{prop}\label{prop:type_less2}
Let $S$ be a numerical semigroup with $\type(S) \le 2$.  Then $S$ is $\lambda$-minimal.
\end{prop}
Proposition 1.1 in \cite{MarzuolaMiller} implies that if $\type(S) =1$, then $\Pa(S) =1$.  Theorem 2.5 in \cite{chen2023enumeratingnumericalsetsassociated} says that if $\type(S) = 2$, then $\Pa(S) =2$.  Therefore, Proposition \ref{prop:type_less2} follows from the next result.  
\begin{prop}\label{prop:PS2_lam_min}
Let $S$ be a numerical semigroup with $\Pa(S) \le 2$.  Then $S$ is $\lambda$-minimal.
\end{prop}
\begin{proof}
If $\Pa(S) =1$, then $\lambda(S)$ is the only partition in $\mathcal{P}(S)$, so it is clearly the smallest one.  Suppose $\Pa(S) = 2$.  Then \cite[Proposition 11]{CONSTANTIN201799} implies that $\lambda(S)$ is not self-conjugate. So, the two partitions in $\mathcal{P}(S)$ are $\lambda(S)$ and its conjugate, which clearly have the same size.
\end{proof}

The main goal of this section is to prove Theorem \ref{thm:type3minimality}.  We will also give some remarks on the $\lambda$-minimality of numerical semigroups of type $4$.

We recall the discussion of the dual of a numerical set from Section \ref{sec:void}. If $T$ is a numerical set, then the dual of $T$ is 
\[
T^* = \{x\in \Z \colon \F(T) - x\not\in T\}.
\]
Proposition~\ref{prop:enum_Tdual} says that $\A(T) = \A(T^*)$. Let $S=\A(T)$. Therefore, $T = S \cup I$ and $T^* = S \cup J$ for some order ideals $I, J$ of $(\M(S),\preccurlyeq)$.  We now explain how these order ideals are related.

\begin{prop}\label{prop:dual_ideal}
    Let $S$ be a numerical semigroup with $\F(S) = F$. Let $T$ be a numerical set such that $T = S\cup I$ for some order ideal $I$ of $(\M(S),\preccurlyeq)$. Then $T^* = S\cup I^*$ where 
    \[
    I^* = \{x \in \M(S) \colon F - x \notin I\}.
    \]
    Moreover, $I^*$ is an order ideal of $(\M(S),\preccurlyeq)$.
\end{prop}

\begin{proof}\label{i_star}
Note that
\begin{align*}
T^*&=\{x\in\mathbb{Z}: F-x\notin S\cup I\}\\
&=\{x\in \M(S): F-x\notin S\cup I\} \cup \{x\in\mathbb{Z}: x\not\in \M(S), F-x\notin S\cup I\}.
\end{align*}
If $x\in \M(S)$, then $F-x\notin S$. So we know that
\[\{x\in \M(S): F-x\notin S\cup I\}=\{x\in \M(S): F-x\notin  I\}=I^*.\]
Moreover, if $x\notin \M(S)$, then $x\in S$ or $F-x\in S$. This means that if $x\notin \M(S)$ and $F-x\notin S\cup I$, then $x\in S$. Also, if $x\in S$, we automatically have that $x\notin \M(S)$ and $F-x\notin S\cup I$.
Therefore,
\[\{x\in\mathbb{Z}: x\not\in \M(S), F-x\notin S\cup I\}
=S.\]

    Now, it is left to see that $I^*$ is an order ideal of $(\M(S),\preccurlyeq)$. Suppose $x \in I^*$ and $x \preceq y$. Then $F-x\notin I$ and $F-y\preccurlyeq F-x$. Since $I$ is an order ideal, this implies that $F-y \notin I$. So $y \in I^*$, and $I^*$ is an order ideal. 
\end{proof}

The main result of \cite[Section 6]{chen2023enumeratingnumericalsetsassociated} is to determine $\Pa(S)$ for any numerical semigroup $S$ with $\type(S) = 3$.  We build on these results to show that all of these semigroups are $\lambda$-minimal. 
For the remainder of this section, $S$ is a numerical semigroup of type $3$ where $\PF(S) = \set{ P, Q, F}$ with $P<Q<F$.

\begin{lemma}\cite[Lemma 6.2]{chen2023enumeratingnumericalsetsassociated}
    \label{type 3 p(s) 2}
    If $P+Q-F\in S$ and $Q-P \notin \M(S)$, then $\Pa(S) = 2$.
\end{lemma}
Proposition \ref{prop:type_less2} implies that numerical semigroups of this kind are $\lambda$-minimal.

\begin{lemma}\cite[Lemma 6.1]{chen2023enumeratingnumericalsetsassociated}
    \label{type 3 4 self dual}
    If $P+Q-F \notin S$, then $\Pa(S)=4$. Moreover, $(\M(S),\preccurlyeq)$ has $4$ order ideals, each of which is self-dual.
\end{lemma}
Proposition \ref{self-dual-implies-lambda-S-small} implies that numerical semigroups of this kind are $\lambda$-minimal.

\begin{lemma}\cite[Lemma 6.4]{chen2023enumeratingnumericalsetsassociated}
    \label{type 3 ideals case}
    Suppose $P+Q-F \in S$ and $Q-P \in \M(S)$. If $S\cup I$ is a numerical set associated to $S$, then $I$ is $\varnothing$, $\M(S)$, or one of the following:   
    \begin{eqnarray*}
    I_1 & = &  \set{x\in \M(S) \colon Q-P \preccurlyeq x}, \\
    I_2 & = & \set{x\in \M(S) \colon F-Q \preccurlyeq x}.
    \end{eqnarray*}
    (Note that it is possible to have $I_1 = I_2$.) 
\end{lemma}

\begin{lemma}\cite[Section 6]{chen2023enumeratingnumericalsetsassociated}
    \label{type 3 ideals duals}
    For the order ideals $I_1$ and $I_2$ of Lemma \ref{type 3 ideals case}, we have that $I_1^* = I_2$.
\end{lemma}

\begin{prop}
    If $P+Q-F\in S$ and $Q-P \in \M(S)$, then $S$ is $\lambda$-minimal.
\end{prop}

\begin{proof}
    Since the ideals $\varnothing$ and $\M(S)$ correspond to $\lambda(S)$ and its conjugate, we may disregard these. We claim that $|\lambda(S)| \leq |\lambda(S\cup I_1)|$. By Lemma \ref{set_counting}, it suffices to show that $|B|\leq |A|$, where
    \begin{eqnarray*}    
    A & = & \{(i,g):i\in I_1,\ g\notin S\cup I_1,\ i<g\},\\
    B & = & \{(s,i):s\in S,\ i\in I_1,\ s<i\}.
    \end{eqnarray*}
    Consider the map $f: B\to A$ defined by
    $$(s, i)\mapsto(Q-i, Q-s).$$
    To confirm that this map is well-defined, we must show that $Q-i \in I_1$ and $Q-s \notin S\cup I_1$. 

    In order to show $Q-i\in I_1$, we first show that $Q-i\in \M(S)$.  Note that $Q-(Q-P)=P\notin S$, which means that $Q-P\not\preccurlyeq Q$. Now, since $i \in I_1$ we see that
    $Q-P\preccurlyeq i$. It follows that $i\not\preccurlyeq Q$, that is $Q-i\notin S$.
    Since $P$, $Q$ are the maximal elements of $(\M(S),\preccurlyeq)$ and $i\not\preccurlyeq Q$, we see that $i\preccurlyeq P$. Therefore,
    \[
    (Q-i)-(Q-P)=P-i\in S.
    \]
    Now, if $F-(Q-i)\in S$, then
    \[
    F-(Q-P)= \left(F-(Q-i)\right) + \left((Q-i)-(Q-P)\right) \in S.
    \]
    However, this contradicts the fact that $Q-P\in \M(S)$. Therefore, $F-(Q-i)\notin S$, and so $Q-i\in \M(S)$.
    We have shown that $Q-i\in\M(S)$ and $Q-P \preccurlyeq Q-i$. It follows that $Q-i\in I_1$.

    We now show that $Q-s \notin S\cup I_1$. Note that $Q-s \notin S$ as otherwise $Q=(Q-s)+s\in~S$. If $Q-s \in I_1$, then $Q-P \preccurlyeq Q-s$. This means that $(Q-s) - (Q-P) = P-s\in S$. But $(P-s)
    +s=P\notin S$, so this is a contradiction. Hence, $Q-s \notin S \cup I_1$.
    
    Since $f$ is trivially injective, we conclude that $|B| \leq |A|$. So $|\lambda(S)| \leq |\lambda(S\cup I_1)|$. By Lemma \ref{type 3 ideals duals} and Proposition \ref{prop:enum_Tdual}, $|\lambda(S\cup I_1)|= |\lambda(S\cup I_2)|$. Hence, $S$ is $\lambda$-minimal. 
\end{proof}
\noindent Collecting these results completes the proof of Theorem \ref{thm:type3minimality}.

We conclude by discussing numerical semigroups of type $4$. Theorem \ref{thm:inf_non_lam} gives an infinite family of semigroups that are not $\lambda$-minimal.  Lemma \ref{Frob} implies that each of these semigroups has type $4$. 

Consider the staircase family numerical semigroups 
    \[
    S_{5, k, 4}=\set{0,5,10, ..., 5k, 5(k+1) \rightarrow}.
    \]
    By Theorem \ref{thm:staircase}, these semigroups are $\lambda$-minimal. Lemma \ref{lem:staircase} implies that for any positive integer $k,\ \type(S_{5, k, 4})=4$.
Thus, there are also infinitely many semigroups of type $4$ that are $\lambda$-minimal.

\bibliographystyle{amsplain}
\nocite{*}
\bibliography{bib.bib}

\end{document}